%% file: main_noETDS.tex
\documentclass{amsart}
\usepackage{amssymb,enumerate,amsmath,amsfonts,mathtools,multicol}
\usepackage{graphics,float,pdfpages,subcaption,tikz-cd}
\usepackage[shortlabels]{enumitem}
\input{defs}  
\usepackage[style=alphabetic,natbib=true,backend=biber,giveninits=true,minalphanames=3,maxalphanames=4,maxbibnames=99]{biblatex}
\setenumerate{nosep}

\begin{document}
\title{Symbolic factors of $\cS$-adic subshifts of finite alphabet rank}
\author{Basti\'an Espinoza}
\address{Departamento de Ingenier\'{\i}a Matem\'atica and Centro de Modelamiento Matem\'atico, Universidad de Chile, Beauchef 851, Santiago, Chile.}
\address{Laboratoire Amiénois de Mathématiques Fondamentales et Appliquées, CNRS-UMR 7352, Université de Picardie Jules Verne, 33 rue Saint Leu, 80039 Amiens cedex 1, France.}

\email{bespinoza@dim.uchile.cl}

\maketitle

\begin{abstract}
This paper studies several aspects of symbolic ({\em i.e.}\ subshift) factors of $\mathcal{S}$-adic subshifts of finite alphabet rank. First, we address a problem raised in \cite{DDPM20} about the topological rank of symbolic factors of $\mathcal{S}$-adic subshifts and prove that this rank is at most the one of the extension system, improving results from \cite{E20} and \cite{GH2020}. As a consequence of our methods, we prove that finite topological rank systems are coalescent. Second, we investigate the structure of fibers $\pi^{-1}(y)$ of factor maps $\pi\colon(X,T)\to(Y,S)$ between minimal $\cS$-adic subshifts of finite alphabet rank and show that they have the same finite cardinality for all $y$ in a residual subset of $Y$. Finally, we prove that the number of symbolic factors (up to conjugacy) of a fixed subshift of finite topological rank is finite, thus extending Durand's similar theorem on linearly recurrent subshifts \cite{durand_2000}.
\end{abstract}

\markboth{Basti\'an Espinoza}{On symbolic factors of $\cS$-adic subshifts of finite alphabet rank}


\section{Introduction}

An ordered Bratteli diagram is an infinite directed graph $B = (V,E,\leq)$ such that the vertex set $V$ and the edge set $E$ are partitioned into levels $V= V_0\cup V_1\cup\dots$, $E = E_0\cup\dots$ so that $E_n$ are edges from $V_{n+1}$ to $V_n$, $V_0$ is a singleton, each $V_n$ is finite and $\leq$ is a partial order on $E$ such that two edges are comparable if and only if they start at the same vertex. The order $\leq$ can be extended to the set $X_B$ of all infinite paths in $B$, and the Vershik action $V_B$ on $X_B$ is defined when $B$ has unique minimal and maximal infinite paths with respect to $\leq$.  We say that $(X_B,V_B)$ is a BV representation of the Cantor system $(X,T)$ if both are conjugate. Bratteli diagrams are a tool coming from $C^*$-algebras that, at the beginning of the 90', Herman {\em et. al.} \cite{herman} used to study minimal Cantor systems. Their success at characterizing the strong and weak orbit equivalence for systems of this kind marked a milestone in the theory that motivated many posterior works. Some of these works focused on studying with Bratteli diagrams specific classes of systems and, as a consequence, many of the classical minimal systems have been characterized as Bratteli-Vershik systems with a specific structure. 
Some examples include odometers as those systems that have a BV representation with one vertex per level, substitutive subshifts as {\em stationary} BV (all levels are the same) \cite{DHS99}, certain Toeplitz sequences as ``equal row-sum'' BV \cite{toeplitz}, and (codings of) interval exchanges as BV where the diagram codifies a path in a Rauzy graph \cite{intervals_exchanges_are_bratteli}.
Now, almost all of these examples share certain coarse dynamical behavior: they have finitely many ergodic measures, are not strongly mixing, have zero entropy, are subshifts, and their BV representations have a bounded number of vertices per level, among many others. It turns out that just having a BV representation with a bounded number of vertices per level (or, from now on, having {\em finite topological rank}) implies the previous properties (see, for example, \cite{invariant_measures}, \cite{DM}). Hence, the finite topological rank class arises as a possible framework for studying minimal subshifts and proving general theorems.

This idea has been exploited in many works: Durand {\em et. al.}, in a series of papers (being \cite{eigenvalues} the last one), developed techniques from the well-known substitutive case and obtained a criteria for any BV of finite topological rank to decide if a given complex number is a continuous or measurable eigenvalue, Bezugly {\em et. al.} described in \cite{invariant_measures} the simplex of invariant measures together with natural conditions for being uniquely ergodic, Giordano {\em et. al.} bounded the rational rank of the dimension group by the topological rank (\cite{herman}, \cite{dimension_group_bounded}), among other works. It is important to remark that these works were inspired by or first proved in the substitutive case.

Now, since Bratteli-Vershik systems with finite topological rank at least two are conjugate to a subshift \cite{DM}, it is interesting to try to define them directly as a subshift. This can be done by codifying the levels of the Bratteli diagram as substitutions and then iterate them to obtain a sequence of symbols defining a subshift conjugate to the initial BV system. This procedure also makes sense for arbitrary nested sequences of substitutions (called {\em directive sequences}), independently from the Bratteli diagram and the various additional properties that its codifying substitutions have. Subshifts obtained in this way are called $\cS$-adic (substitution-adic) and may be non-minimal (see for example \cite{BSTY17}). 

Although there are some open problems about finite topological rank systems depending directly on the combinatorics of the underlying Bratteli diagrams, others are more naturally stated in the $\cS$-adic setting ({\em e.g.}, when dealing with endomorphisms, it is useful to have the Curtis–Hedlund–Lyndon Theorem) and, hence, there exists an interplay between $\cS$-adic subshifts and finite topological rank systems in which theorems and techniques obtained for one of these classes can sometimes be transferred to the other. The question about which is the exact relation between these classes has been recently addressed in \cite{DDPM20} and, in particular, the authors proved:
\begin{theo}[\cite{DDPM20}]\label{theo:FTR_from_DDPM}
A minimal subshift $(X,T)$ has topological rank at most $K$ if and only if it is generated by a proper, primitive and recognizable $\cS$-adic sequence of alphabet rank at most $K$.
\end{theo}

In this context, a fundamental question is the following:
\begin{ques}\label{ques:is_factor_FTR}
Are subshift factors of finite topological rank systems of finite topological rank?
\end{ques}

Indeed, the topological rank controls various coarse dynamical properties (number of ergodic measures, rational rank of dimension group, among others) which cannot increase after a factor map, and we also know that big subclasses of the finite topological rank class are stable under symbolic factors, such as the linearly recurrent and the non-superlineal complexity classes \cite{DDPM20}, so it is expected that this question has an affirmative answer. However, when trying to prove this using Theorem \ref{theo:FTR_from_DDPM}, we realize that the naturally inherited $\cS$-adic structure of finite alphabet rank that a symbolic factor has is never recognizable. Moreover, this last property is crucial for many of the currently known techniques to handle finite topological rank systems (even in the substitutive case it is a deep and fundamental theorem of Mossé), so it is not clear why it would be always possible to obtain this property while keeping the alphabet rank bounded or why recognizability is not connected with a dynamical property of the system. Thus, an answer to this question seems to be fundamental to the understanding of the finite topological rank class.

This question has been recently addressed, first in \cite{E20} by purely combinatorial methods, and then also in \cite{GH2020} in the BV formulation by using an abstract construction from \cite{AEG2015}. In this work, we refine both approaches and obtain, as a first consequence, the optimal answer to Question \ref{ques:is_factor_FTR} in a more general, non-minimal context:
\begin{theo}\label{theo:intro:1}
Let $(X,T)$ be an $\cS$-adic subshift generated by an everywhere growing and proper directive sequence of alphabet rank equal to $K$, and $\pi\colon(X,T)\to(Y,S)$ be an aperiodic subshift factor. Then, $(Y,S)$ is an $\cS$-adic subshift generated by an everywhere growing, proper and recognizable directive sequence of alphabet rank at most $K$.
\end{theo}
Here, a directive sequence $\boldsymbol{\sigma} = (\sigma_n\colon\cA_{n+1}^+\to\cA_n^+)_{n\in\N}$ is {\em everywhere growing} if $\lim_{n\to\infty}\min_{a\in\cA_n}|\sigma_0\dots\sigma_{n-1}(a)| = \infty$, and a system $(X,T)$ is {\em aperiodic} if every orbit $\{T^nx:n\in\Z\}$ is infinite. Theorem \ref{theo:intro:1} implies that the topological rank cannot increase after a factor map (Corollary \ref{cor:question_1}). Theorem \ref{theo:intro:1} implies the following sufficient condition for a system to be of finite topological rank:
\begin{cor}\label{cor:factor_nonproper}
Let $(X,T)$ be an aperiodic minimal $\cS$-adic subshift generated by an everywhere growing directive sequence of finite alphabet rank. Then, the topological rank of $(X,T)$ is finite.
\end{cor}
An interesting corollary of the underlying construction of the proof of Theorem \ref{theo:intro:1} is the coalescence property for this kind of systems, in the following stronger form:
\begin{cor}\label{theo:intro:2}
Let $(X,T)$ be an $\cS$-adic subshift generated by an everywhere growing and proper directive sequence of alphabet rank equal to $K$, and $(X,T) \overset{\pi_1}{\to} (X_1,T_1)\overset{\pi_2}{\to}\dots \overset{\pi_L}{\to}(X_L,T_L)$ be a chain of aperiodic subshift factors. If $L > \log_2K$, then at least one $\pi_j$ is a conjugacy.
\end{cor}

One of the results in \cite{durand_2000} is that factor maps between aperiodic linearly recurrent subshifts are finite to one. In particular, they are {\em almost $k$-to-1} for some finite $k$. For finite topological rank subshifts, we prove:
\begin{theo}\label{theo:intro:3}
Let $\pi\colon(X,T)\to(Y,S)$ be a factor map between aperiodic minimal subshifts. Suppose that $(X,T)$ has topological rank equal to $K$. Then $\pi$ is almost $k$-to-1 for some $k \leq K$.
\end{theo}
We use this theorem, in Corollary \ref{cor:cantor_factors}, to prove that Cantor factors of finite topological rank subshifts are either odometers or subshifts.

In \cite{durand_2000}, the author proved that {\em linearly recurrent} subshifts have finite topological rank, and that this kind of systems have finitely many aperiodic subshifts factors up to conjugacy. Inspired by this result, we use ideas from the proof of Theorem \ref{theo:intro:1} to obtain:
\begin{theo}\label{theo:intro:4}
Let $(X,T)$ be a minimal subshift of topological rank $K$. Then, $(X,T)$ has at most $(3K)^{32K}$ aperiodic subshift factors up to conjugacy.
\end{theo}

\smallbreak
Altogether, these results give a rough picture of the set of totally disconnected factors of a given finite topological rank system: they are either equicontinuous or subshifts satisfying the properties in Theorems \ref{theo:intro:1}, \ref{theo:intro:2}, \ref{theo:intro:4} and \ref{theo:intro:3}. Now, in a topological sense, totally disconnected factors of a given system $(X,T)$ are ``maximal'', so, the natural next step in the study of finite topological rank systems is asking about the connected factors. As we have seen, the finite topological rank condition is a rigidity condition. By this reason, we think that the following question has an affirmative answer:
\begin{ques}\label{ques:is_connected_equi}
Let $(X,T)$ be a  minimal system of finite topological rank and $\pi\colon(X,T)\to(Y,S)$ be a factor map. Suppose that $Y$ is connected. Is $(Y,S)$ an equicontinuous system?
\end{ques}
We remark that the finite topological rank class contains all minimal subshifts of non-superlinear complexity \cite{DDPM20}, but even for the much smaller class of linear complexity subshifts the author is not aware of results concerning Question \ref{ques:is_connected_equi}.

\subsection{Organization}
In the next section we give the basic background in topological and symbolic dynamics needed in this article. Section \ref{sec:words_on_combinatorics} is devoted to prove some combinatorial lemmas. The main results about the topological rank of factors are stated and proved in Section \ref{sec:rank_of_factors}. Next, in Section \ref{sec:fibers}, we prove Theorem \ref{theo:intro:3}, which is mainly a consequence of the so-called Critical Factorization Theorem. Finally, in Section \ref{sec:number_factors}, we study the problem about the number of symbolic factors and prove Theorem \ref{theo:intro:4}.

\section{Preliminaries}\label{sec:preliminaries}

For us, the set of natural numbers starts with zero, {\em i.e.}, $\N=\{0,1,2,\dots\}$.

\subsection{Basics in topological dynamics}\label{subsec:topdyn}
A {\em topological dynamical system} (or just a system) is a pair $(X,T)$ where $X$ is a compact metric space and  $T\colon X \to X$ is a {\em homeomorphism} of $X$. We denote by $\text{Orb}_T(x)$ the orbit $\{T^nx : n\in \Z\}$ of $x \in X$. A point $x\in X$ is $p$-{\em periodic} if $T^px=x$, {\em periodic} if it is $p$-periodic for some $p\geq1$ and {\em aperiodic} otherwise. A topological dynamical system is {\em aperiodic} if any point $x\in X$ is aperiodic, is {\em minimal} if the orbit of every point is dense in $X$, and is Cantor if $X$ is a Cantor space (\emph{i.e} $X$ is totally disconnected and does not have isolated points). We use the letter $T$ to denote the action of a topological dynamical system independently of the base set $X$. The {\em hyperspace} of $(X,T)$ is the system $(2^X,T)$, where $2^X$ is the set of all closed subsets of $X$ with the topology generated by the Hausdorff metric $d_H(A,B) = \max(\sup_{x\in A}d(x,A),\ \sup_{y\in B}d(y,A))$, and $T$ the action $A\mapsto T(A)$.

A {\em factor} between the topological dynamical systems $(X,T)$ and $(Y,T)$ is a continuous function $\pi$ from $X$ onto $Y$ such that $\pi\circ T=T\circ \pi$. We use the notation $\pi\colon (X,T)\to (Y,T)$ to indicate the factor. A factor map $\pi\colon (X,T) \to (Y,T)$ is {\em almost K-to-1} if $\#\pi^{-1}(y) = K$ for all $y$ in a residual subset of $Y$. We say that $\pi$ is \emph{distal} if whenever $\pi(x) = \pi(x')$ and $x\not=x'$, we have $\inf_{k\in\Z}\mathrm{dist}(T^kx,T^kx') > 0$.

Given a system $(X,T)$, the {\em Ellis semigroup} $E(X,T)$ associated to $(X,T)$ is defined as the closure of $\{x\mapsto T^nx:n\in\Z\} \subseteq X^X$ in the product topology, where the semi-group operation is given by the composition of functions. On $X$ we may consider the $E(X,T)$-action given by $x\mapsto ux$. Then, the closure of the orbit under $T$ of a point $x\in X$ is equal to the orbit of $x$ under $E(X,T)$. If $\pi\colon (X,T) \to (Y,T)$ is a factor between minimal systems, then $\pi$ induces a surjective map $\pi^*\colon E(X,T) \to E(Y,T)$ which is characterized by the formula
\begin{equation*}
\pi(ux) = \pi^*(u)\pi(x)\quad \text{for all $u \in E(X,T)$ and $x\in X$.}
\end{equation*}
If the context is clear, we will not distinguish between $u$ and $\pi^*(u)$. When $u \in E(2^X,T)$, we write $u\circ A$ instead of $uA$, the last symbol being reserved to mean $uA = \{ux : x \in A\}$. We can describe more explicitly $u\circ A$ as follows: it is the set of all $x\in X$ for which we can find nets $x_\lambda \in A$ and $m_\lambda\in\Z$ such that $\lim_\lambda T^{m_\lambda} x_\lambda = x$ and $\lim_\lambda T^{m_\lambda} = u$. Finally, we identify $X$ with $\{\{x\} \subseteq 2^X : x\in X\}$, so that the restriction map $E(2^X,T) \to E(X,T)$ which sends $u\in E(2^X,T)$ to the restriction $u|_X\colon X\to X$ is an onto morphism of semigroups. As above, we will not distinguish between $u\in 2^X$ and $u|_X$.

\subsection{Basics in symbolic dynamics}\label{subsec:symbdyn}
\subsubsection{Words and subshifts} \label{subsubsec:subshifts}

Let ${\mathcal A}$ be an {\em alphabet} {\em i.e.} a finite set. Elements in ${\mathcal A}$ are called {\em letters} and concatenations $w = a_1\dots a_\ell$ of them are called words. The number $\ell$ is the length of $w$ and it is denoted by $|w|$, the set of all words in $\cA$ of length $\ell$ is $\cA^\ell$, and $\cA^+ = \bigcup_{\ell\geq1}\cA^\ell$.
The word $w\in \cA^+$ is $|u|$-{\em periodic}, with $u\in\cA^+$, if $w$ occurs in a word of the form $uu\dots u$. We define $\per(w)$ as the smallest $p$ for which $w$ is $p$-periodic. 
We will use notation analogous to the one introduced in this paragraph when dealing with infinite words $x \in \cA^\N$ and bi-infinite words $x \in \cA^\Z$. 
The set $\cA^+$ equipped with the operation of concatenation can be viewed as the free semigroup on $\cA$. 
It is convenient to introduce the empty word $1$, which has length $0$ and is a neutral element for the concatenation. In particular, $\cA^+\cup\{1\}$ is the free monoid in $\cA$. Finally, for $\cW\subseteq\cA^+$, we write $\langle \cW\rangle \coloneqq \min_{w\in \cW}|w|$ and $|W| \coloneqq \max_{w\in \cW}|w|$.

The {\em shift map} $T \colon {\mathcal A}^{\mathbb Z} \to {\mathcal A}^{\mathbb Z}$ is defined by $T ((x_n)_{n\in \mathbb{Z}}) = (x_{n+1})_{n\in \mathbb{Z}}$. For $x \in \cA^\Z$ and integers $i < j$, we denote by $x_{[i,j)}$ the word $x_ix_{i+1}\dots x_j$. Analogous notation will be used when dealing with intervals of the form $[i,\infty)$, $(i,\infty)$, $(-\infty,i]$ and $(-\infty,i)$. A {\em subshift} is a topological dynamical system $(X,T)$ where $X$ is a closed and $T$-invariant subset of ${\mathcal A}^{\mathbb Z}$ (we consider the product topology in ${\mathcal A}^{\mathbb Z}$) and $T$ is the shift map. Classically one identifies $(X,T)$ with $X$, so one says that $X$ itself is a subshift. When we say that a sequence  in a subshift is periodic (resp. aperiodic), we implicitly mean that this sequence is periodic (resp. aperiodic) for the action of the shift. Therefore, if $x\in \cA^\Z$ is periodic, then $\per(x)$ is equal to the size of the orbit of $x$. 
The {\em language} of a subshift $X\subseteq\cA^\Z$ is the set $\cL(X)$ of all words $w\in\cA^+$ that occur in some $x\in X$.

The pair $(x,\tilde{x}) \in \cA^\Z\times \cA^\Z$ is {\em right asymptotic} if there exist $k \in \Z$ satisfying $x_{(k,\infty)} = \tilde{x}_{(k,\infty)}$ and $x_k\not=\tilde{x}_k$. If moreover $k=0$, $(x,\tilde{x})$ is a {\em centered right asymptotic}. A {\em right asymptotic tail} is an element $x_{(0,\infty)}$, where $(x,\tilde{x})$ is a centered right asymptotic pair. We make similar definitions for left asymptotic pairs and tails.

\subsubsection{Morphisms and substitutions}\label{subsubsec:morphisms}

Let $\cA$ and $\cB$ be finite alphabets and $\tau \colon \cA^+ \to \cB^+$ be a morphism between the free semigroups that they define. Then, $\tau$ extends naturally to maps from $\cA^\mathbb{N}$ to itself and from $\cA^\mathbb{Z}$ to itself  in the obvious way by concatenation (in the case of a twosided sequence we apply $\tau$ to positive and negative coordinates separately and we concatenate the results at coordinate zero).
We say that $\tau$ is {\em positive} if for every $a\in\cA$, all letters $b \in \cB$ occur in $\tau(a)$, is {\em $r$-proper}, with $r\geq1$, if there exist $u,v \in \cB^r$ such that $\tau(a)$ starts with $u$ and ends with $v$ for any $a \in \cA$, is {\em proper} when is $1$-proper, and is {\em letter-onto} if for every $b \in \cB$ there exists $a \in \cA$ such that $b$ occurs in $a$. The minimum and maximum length of $\tau$ are, respectively, the numbers $\langle\tau\rangle \coloneqq \langle\tau(\cA)\rangle = \min_{a\in\cA}|\tau(a)|$ and $|\tau| \coloneqq |\tau(\cA)| = \max_{a\in\cA}|\tau(a)|$.

We observe that any map $\tau\colon \cA\to \cB^+$ can be naturally extended to a morphism (that we also denote by $\tau$) from $\cA^+$ to $\cB^+$ by concatenation, and any morphism $\tau\colon \cA^+\to\cB^+$ is uniquely determined by its restriction to $\cA$. From now on, we will use the same notation for denoting a map $\tau\colon\cA\to\cB^+$ and its extension to a morphism $\tau\colon\cA^+\to\cB^+$.

\begin{defi}\label{defi:factorizations}
Let $X\subseteq \cA^\Z$ be a subshift and $\sigma\colon \cA^+\to \cB^+$ be a morphism. We say that $(k,x) \in \Z\times X$ is a {\em $\sigma$-factorization of $y \in \cB^\Z$ in $X$} if $y = T^k\sigma(x)$. If moreover $k\in[0,|\sigma(x_0)|)$, then $(k,x)$ is a {\em centered $\sigma$-factorization in $X$}. 

The pair $(X,\sigma)$ is {\em recognizable} if every point $y \in \cB^\Z $ has at most one centered $\sigma$-factorization in $X$, and {\em recognizable with constant $r \in \N$} if whenever $y_{[-r,r]} = y'_{[-r,r]}$ and $(k,x)$, $(k',x')$ are centered $\sigma$-factorizations of $y, y' \in \cB^\Z$ in $X$, respectively, we have $(k,x_0) = (k',x'_0)$.

The {\em cuts} of $(k,x)$ are defined by
\begin{equation*}
c_{\sigma,j}(k,x) = \begin{cases}
	-k + |\sigma(x_{[0,j)})| &\text{if }j \geq 0, \\
	-k - |\sigma(x_{[j,0)})| &\text{if }j < 0. \end{cases}
\end{equation*}
We write $C_\sigma(k,x) = \{c_{\sigma,j}(k,x) : j \in \Z\}$.
\end{defi}
\begin{rem}\label{rem:factorizations}
In the context of the previous definition:
\begin{enumerate}[label=(\roman*)]
\item\label{rem:factorizations:image_subshift} The point $y \in \cB^\Z$ has a (centered) $\sigma$-factorization in $X$ if and only if $y$ belongs to the subshift $Y \coloneqq \bigcup_{n\in\Z}T^n\sigma(X)$. Hence, $(X,\sigma)$ is recognizable if and only if every $y\in Y$ has a exactly {\em one} centered $\sigma$-factorization in $X$.

\item\label{rem:factorizations:centered} If $(k,x)$ is a $\sigma$-factorization of $y \in \cB^\Z$ in $X$, then $(c_{\sigma,j}(k,x), T^j x)$ is a $\sigma$-factorization of $y$ in $X$ for any $j \in \Z$. There is exactly one factorization in this class that is centered.

\item If $(X,\sigma)$ is recognizable, then it is recognizable with constant $r$ for some $r \in \N$ \cite{DDPM20}.
\end{enumerate}
\end{rem}

The behavior of recognizability under composition of morphisms is given by the following lemma.
\begin{lem}[\cite{BSTY17}, Lemma 3.5]\label{lem:recognizability_of_composition}
Let $\sigma\colon \cA^+\to \cB^+$ and $\tau\colon\cB^+\to \cC^+$ be morphisms, $X\subseteq \cA^\Z$ be a subshift and $Y = \bigcup_{k\in\Z}T^k\sigma(X)$. Then, $(X, \tau\sigma)$ is recognizable if and only if $(X,\sigma)$ and $(Y,\tau)$ are recognizable.
\end{lem}

Let $X \subseteq \cA^\Z$ and $Z \subseteq \cC^\Z$ be subshifts and $\pi\colon(X,T) \to (Z,T)$ a factor map. The classic Curtis–Hedlund–Lyndon Theorem asserts that $\pi$ has a {\em local code}, this is, a function $\psi\colon \cA^{2r+1}\to\cC$, where $r\in\N$, such that $\pi(x) = (\psi(x_{[i-r,i+r]}))_{i\in\Z}$ for all $x \in X$. The integer $r$ is called the a radius of $\pi$. The following lemma relates the local code of a factor map to proper morphisms.
\begin{lem}\label{lem:factor_to_morphism}
Let $\sigma\colon\cA^+\to\cB^+$ be a morphism, $X \subseteq \cA^\Z$ and $Z\subseteq \cC^\Z$ be subshifts, and $Y = \bigcup_{k\in\Z}T^k\sigma(X)$. Suppose that $\pi\colon (Y,T)\to(Z,T)$ is a factor map of radius $r$ and that $\sigma$ is $r$-proper. Then, there exists a proper morphism $\tau\colon\cA^+\to\cC^+$ such that $|\tau(a)| = |\sigma(a)|$ for any $a \in \cA$, $Z = \bigcup_{k\in\Z}T^k\tau(X)$ and the following diagram commutes:
\begin{equation}\label{diag:lem:factor_to_morphism}
\begin{tikzcd}
	X \arrow{d}[swap]{\sigma} \arrow{dr}{\tau}& \\
	Y \arrow{r}{\pi} & Z
\end{tikzcd}
\end{equation}
\end{lem}
\begin{proof}
Let $\psi\colon \cA^{2r+1}\to\cB$ be a local code of radius $r$ for $\pi$ and $u,v \in \cB^r$ be such that $\sigma(a)$ starts with $u$ and ends with $v$ for all $a \in \cA$. We define $\tau\colon \cA\to \cC^+$ by $\tau(a) = \psi(v\sigma(a)u)$. Then, since $\sigma$ is $r$-proper, $\tau$ is proper and we have $\pi(\sigma(x)) = \tau(x)$ for all $x \in X$ (this is, Diagram \eqref{diag:lem:factor_to_morphism} commutes). In particular, 
\begin{equation*}
\bigcup_{k\in\Z}T^k\tau(X) = 
\bigcup_{k\in\Z}T^k\pi(\sigma(X)) = 
\pi(Y) = Z.
\end{equation*}
\end{proof}

\subsubsection{$\cS$-adic subshifts}\label{subsubsec:Sadicsubshifts}

We recall the definition of an \emph{$\cS$-adic subshift} as stated in \cite{BSTY17}.
A {\em directive sequence} $\boldsymbol{\sigma} = (\sigma_n \colon \mathcal{A}_{n+1}^+\to \mathcal{A}_n^+)_{n \in \N}$ is a sequence of morphisms. For $0\leq n<N$, we denote by $\sigma_{[n,N)}$ the morphism $\sigma_n \circ \sigma_{n+1} \circ \dots \circ \sigma_{N-1}$. We say that $\boldsymbol{\sigma}$ is {\em everywhere growing} if
\begin{equation}\label{eq:def_everywhere_growing}
\lim_{N\to +\infty}\langle\sigma_{[0,N)}\rangle = +\infty,
\end{equation}
and {\em primitive} if for any $n\in \N$ there exists $N>n$ such that $\sigma_{[n,N)}$ is positive. We remark that this notion is slightly different from the usual one used in the context of substitutional dynamical systems. Observe that $\boldsymbol{\sigma}$ is everywhere growing if $\boldsymbol{\sigma}$ is primitive. Let $\cP$ be a property for morphisms ({\em e.g.} proper, letter-onto, etc). We say that $\boldsymbol{\sigma}$ has property $\cP$ if $\sigma_n$ has property $\cP$ for every $n \in \N$.

For $n\in \N$, we define
\begin{equation*}
X^{(n)}_{\boldsymbol{\sigma}} = \big\{x \in \mathcal{A}_n^\Z :\
\mbox{$\forall \ell\in\N$, $x_{[-\ell,\ell]}$ occurs in $\sigma_{[n,N)}(a)$ for some $N>n, a \in\mathcal{A}_N$}\big\}.
\end{equation*}
This set clearly defines a subshift that we call the {\em $n$th level of the $\cS$-adic subshift generated by $\boldsymbol{\sigma}$}. We set $X_{\boldsymbol{\sigma}} = X^{(0)}_{\boldsymbol{\sigma}}$ and simply call it the \emph{$\cS$-adic subshift} generated by $\boldsymbol{\sigma}$. If $\boldsymbol{\sigma}$ is everywhere growing, then every $X^{(n)}_{\boldsymbol{\sigma}}$, $n\in\N$, is non-empty; if $\boldsymbol{\sigma}$ is primitive, then $X^{(n)}_{\boldsymbol{\sigma}}$ is minimal for every $n\in\N$. There are non-everywhere growing directive sequences that generate minimal subshifts.

The relation between levels of an $\cS$-adic subshift is given by the following lemma.
\begin{lem}[\cite{BSTY17}, Lemma 4.2]\label{lem:desubstitution}
Let $\boldsymbol{\sigma} = (\sigma_n \colon \mathcal{A}_{n+1}^+\to \mathcal{A}_n^+)_{n \in \N}$ be a directive sequence of morphisms. If $0\leq n < N$ and $x \in X_{\boldsymbol{\sigma}}^{(n)}$, then there exists a (centered) $\sigma_{[n,N)}$-factorization in $X_{\boldsymbol{\sigma}}^{(N)}$. In particular, $X_{\boldsymbol{\sigma}}^{(n)} = \bigcup_{k\in\Z}T^k\sigma_{[n,N)}(X_{\boldsymbol{\sigma}}^{(N)})$.
\end{lem}
The levels $X_{\boldsymbol{\sigma}}^{(n)}$ can be described in an alternative way if $\boldsymbol{\sigma}$ satisfies the correct hypothesis.
\begin{lem}\label{lem:eg_and_proper_imply_limitpoints}
Let $\boldsymbol{\sigma} = (\sigma_n\colon\cA_{n+1}^+\to\cA_n^+)_{n\in\N}$ be an everywhere growing and proper directive sequence. Then, for every $n \in \N$,
\begin{equation}\label{eq:lem:eg_and_proper_imply_limitpoints:1}
X_{\boldsymbol{\sigma}}^{(n)} =
\bigcap_{N>n}\bigcup_{k\in\Z}T^k\sigma_{[n,N)}(\cA_N^\Z)
\end{equation}
\end{lem}
\begin{proof}
Let $Z$ be the set in the right-hand side of \eqref{eq:lem:eg_and_proper_imply_limitpoints:1}. Since, by Lemma \ref{lem:desubstitution}, $X_{\boldsymbol{\sigma}}^{(n)} = \bigcup_{k\in\Z}T^k\sigma_{[n,N)}(X_{\boldsymbol{\sigma}}^{(N)})$ for any $N> n$, we have that $X_{\boldsymbol{\sigma}}^{(n)}$ included in $Z$.

Conversely, let $x\in Z$ and $\ell\in\N$. We have to show that $x_{[-\ell,\ell)}$ occurs in $\sigma_{[n,N)}(a)$ for some $N>n$ and $a \in \cA_N$. Let $N>n$ be big enough so that $\sigma_{[n,N)}$ is $\ell$-proper. Then, by the definition of $Z$, there exists $y \in \cA_N^\Z$ such that $x_{[-\ell,\ell)}$ occurs in $\sigma_{[n,N)}(y)$. Since $\langle \sigma_{[n,N)}\rangle \geq \ell$ (as $\sigma_{[n,N)}$ is $\ell$-proper), we deduce that
\begin{equation}\label{eq:lem:eg_and_proper_imply_limitpoints:2}
\text{$x_{[-\ell,\ell)}$ occurs in $\sigma_{[n,N)}(ab)$ for some word $ab$ of length 2 occurring in $y$.}
\end{equation}
Hence, by denoting by $u$ and $v$ the suffix and prefix of length $\ell$ of $\tau_{[n,N)}(a)$ and $\tau_{[n,N)}(b)$, respectively, we have that $x_{[-\ell,\ell)}$ occurs in $\sigma_{[n,N)}(a)$, in $\tau_{[n,N)}(b)$, or in $uv$. In the first two cases, we are done. In the last case, we observe that since $\sigma_{[n,N)}$ is $\ell$-proper, the following is true: for every $M>N$ such that $\langle\sigma_{[N,M)}\rangle \geq 2$, $vu\sqsubseteq\sigma_{[n,M)}(c)$ for any $c \in \cA_M$. In particular, $x_{[-\ell,\ell)}\sqsubseteq\tau_{[n,M)}(c)$ for such $M$ and $c$. We have proved that $x \in X_{\boldsymbol{\sigma}}^{(n)}$.
\end{proof}

We define the {\em alphabet rank} of a directive sequence $\boldsymbol{\tau}$ as 
$$AR(\boldsymbol{\tau}) = \liminf_{n\to +\infty} \#\cA_n.$$

A {\em contraction} of $\boldsymbol{\tau}$ is a sequence $\tilde{\boldsymbol{\tau}} = (\tau_{[n_k,n_{k+1})}\colon \cA_{n_{k+1}}^+\to \cA_{n_k}^+)_{k\in\N}$, where $0 = n_0 < n_1 < n_2 < \dots$. Observe that any contraction of $\boldsymbol{\tau}$ generates the same $\cS$-adic subshift $X_{\boldsymbol{\tau}}$. When the context is clear, we will use the same notation to refer to $\boldsymbol{\tau}$ and its contractions. If $\boldsymbol{\tau}$ has finite alphabet rank, then there exists a contraction $\tilde{\boldsymbol{\tau}} = (\tau_{[n_k,n_{k+1})}\colon \cA_{n_{k+1}}^+\to \cA_{n_k}^+)_{k\in\N}$ of $\boldsymbol{\tau}$ in which $\cA_{n_k}$ has cardinality $AR(\boldsymbol{\tau})$ for every $k\geq1$.

Finite alphabet rank $\cS$-adic subshifts are {\em eventually recognizable}:
\begin{theo}[\cite{DDPM20}, Theorem 3.7]\label{theo:finite_rank_is_expansive}
Let $\boldsymbol{\sigma}$ be an everywhere growing directive sequence of alphabet rank equal to $K$. Suppose that $X_{\boldsymbol{\sigma}}$ is aperiodic. Then, at most $\log_2 K$ levels $(X_{\boldsymbol{\sigma}}^{(n)},\sigma_n)$ are not recognizable.
\end{theo}

We will also need the following property.
\begin{theo}[\cite{EM}, Theorem 3.3]\label{prop:asymptotic_from_EM}
Let $(X,T)$ be an $\cS$-adic subshift generated by an everywhere growing directive sequence of alphabet rank $K$. Then, $X$ has at most $144K^7$ right (resp. left) asymptotic tails.
\end{theo}
\begin{proof}
In the proof of Theorem 3.3 in \cite{EM} the authors show the following: the set consisting of pairs $(x,y) \in X \times X$ such that $x_{(-\infty,0)} = y_{(-\infty,0)}$ and $x_0\not=y_0$ has at most $144K^7$ elements. In our language, this is equivalent to saying that $X$ has at most $144K^7$ left asymptotic tails. Since this is valid for any $\cS$-adic subshift generated by an everywhere growing directive sequence of alphabet rank $K$, $144K^7$ is also an upper bound for right asymptotic tails.
\end{proof}

\section{Combinatorics on words lemmas}\label{sec:words_on_combinatorics}

In this section we present several combinatorial lemmas that will be used throughout the article. 

\subsection{Lowering the rank}\label{subsec:free_rank}
Let $\sigma\colon\cA^+\to\cB^+$ be a morphism. Following ideas from \cite{handbook_words}, we define the {\em rank} of $\sigma$ as the least cardinality of a set of words $\cD \subseteq \cB^+$ such that $\sigma(\cA^+) \subseteq \cD^+$. Equivalently, the rank is the minimum cardinality of an alphabet $\cC$ in a decomposition into morphisms $\cA^+ \overset{q}{\longrightarrow} \cC^+ \overset{p}{\longrightarrow} \cB^+$ such that $\sigma = pq$. In this subsection we prove Lemma \ref{lem:factorize_sequence}, which states that in certain technical situation, the rank of the morphism $\sigma$ under consideration is small and its decomposition $\sigma = pq$ satisfies additional properties.

We start by defining some morphisms that will be used in the proofs of this subsection. If $a\not=b \in \cA$ are different letters and $\tilde{a}$ is a letter not in $\cA$, then we define $\phi_{a,b}\colon\cA^+\to(\cA\setminus\{b\})^+$, $\psi_{a,b}\colon\cA^+\to\cA^+$ and $\theta_{a, \tilde{a}}\colon\cA^+\to(\cA\cup\{\tilde{a}\})^+$ by
\begin{equation*}
\phi_{a,b}(c) = \begin{cases}
	c & \text{if $c\not=b$,} \\
	a & \text{if $c=b$.} \end{cases}
\qquad
\psi_{a,b}(c) = \begin{cases}
	c & \text{if $c\not=b$,} \\
	ab & \text{if $c=b$.} \end{cases}
\qquad
\theta_{a,\tilde{a}}(c) = \begin{cases}
	c & \text{if $c\not=a$,} \\
	\tilde{a}a & \text{if $c=a$.} \end{cases}
\end{equation*}
Observe that these morphisms are letter-onto. Before stating the basic properties of these morphisms, we need one more set of definitions.

For a morphism $\sigma\colon\cA^+\to\cB^+$, we define $|\sigma|_1 = \sum_{a\in\cA}|\sigma(a)|$. When $u,v,w\in\cA^+$ satisfy $w = uv$, we say that $u$ is a prefix of $w$ and that $v$ a suffix of $w$. Recall that $1$ stands for the empty word.
\begin{lem}\label{lem:lower_free_rank_preparation}
Let $\sigma\colon\cA^+\to\cB^+$ be a morphism.
\begin{enumerate}[label=(\roman*)]
\item If $\sigma(a) = \sigma(b)$ for some $a\not=b\in\cA$, then $\sigma = \sigma'\phi_{a,b}$, where $\sigma'\colon(\cA\setminus\{b\})^+\to\cB^+$ is the restriction of $\sigma$ to $(\cA\setminus\{b\})^+$.

\item If $\sigma(a)$ is a prefix of $\sigma(b)$ and $\sigma(b) = \sigma(a)t$ for some non-empty $t\in\cB^+$, then $\sigma = \sigma'\psi_{a,b}$, where $\sigma'\colon\cA^+\to\cB^+$ is defined by
\begin{equation}\label{eq:lem:lower_free_rank_preparation:cut}
\sigma'(c) = \begin{cases}
	\sigma(c) & \text{if $c\not=b$,} \\
	t & \text{if $c=b$.} \end{cases}
\end{equation}

\item If $\sigma(a) = st$ for some $s,t\in\cB^+$ and $a\in\cA$, then $\sigma = \sigma'\theta_{a,\tilde{a}}$, where $\sigma'\colon(\cA\cup\{\tilde{a}\})^+\to\cB^+$ is defined by
\begin{equation}\label{eq:lem:lower_free_rank_preparation:split}
\sigma'(c) = \begin{cases}
	\sigma(c) & \text{if $c\not=a,\tilde{a}$,} \\
	s & \text{if $c=\tilde{a}$,} \\
	t & \text{if $c=a$.} \end{cases}
\end{equation}

\end{enumerate}
\end{lem}
\begin{proof}
The lemma follows from unraveling the definitions. For instance, in case (ii), we have $\sigma'(\psi_{a,b}(a)) = \sigma'(a) = \sigma(a)$, $\sigma'(\psi_{a,b}(b)) = \sigma'(ab) = \sigma(a)t = \sigma(b)$, and $\sigma'(\psi_{a,b}(c)) = \sigma'(c) = \sigma(c)$ for all $c \not= a,b$, which shows that $\sigma'\psi_{a,b} = \sigma$.
\end{proof}

\begin{lem}\label{lem:lower_free_rank_s_equal_0}
Let $\{\sigma_j\colon\cA^+\to\cB_j^+\}_{j\in J}$ be a set of morphisms such that
\begin{equation}\label{eq:lower_free_rank_s_equal_0:1}
\text{for every fixed $a \in \cA$, $\ell_a \coloneqq |\sigma_j(a)|$ is constant for any chosen $j \in J$,}
\end{equation}
and $u,v \in \cA^+$, with $u$ of length at least $\ell \coloneqq \sum_{a\in\cA}\ell_a$. Assume that $u$ and $v$ start with different letters and that $\sigma_j(u)$ is a prefix of $\sigma_j(v)$ for every $j \in J$. 

Then, there exist a letter-onto morphism $q\colon\cA^+ \to \cC^+$, with $\#\cC < \#\cA$, and morphisms $\{p_j\colon\cC^+\to\cB_j^+\}_{j\in J}$ satisfying a condition analogous to \eqref{eq:lower_free_rank_s_equal_0:1} and such that $\sigma_j = p_jq$.
\end{lem}

\begin{rem}\label{rem:lem:lower_free_rank_s_equal_0}
If in the previous lemma we change the last hypothesis to ``$u$ and $v$ {\em end} with different letters and $\sigma_j(u)$ is a {\em suffix} of $\sigma_j(v)$ for every $j \in J$", then the same conclusion holds. This observation will be used in the proof of Lemma \ref{lem:not_distal_implies_recursion}.
\end{rem}
\begin{proof}[Proof (of Lemma \ref{lem:lower_free_rank_s_equal_0})]
By contradiction, we assume that $u$, $v$ and $\{\sigma_j\}_{j \in J}$, are counterexamples for the lemma. Moreover, we suppose that $\ell$ is as small as possible. 

Let us write $u = au'$ and $v=bv'$, where $a,b \in \cA$. Since $\sigma_j(u)$ is a prefix of $\sigma_j(v)$, we have that for every $j \in J$,
\begin{equation}\label{eq:lower_free_rank_s_equal_0:2}
\text{one of the words in $\{\sigma_j(a), \sigma_j(b)\}$ is a prefix of the other.}
\end{equation}

We consider two cases. First, we suppose that $\ell_a = \ell_b$. In this case, \eqref{eq:lower_free_rank_s_equal_0:2} implies that $\sigma_j(a) = \sigma_j(b)$ for every $j\in J$. Hence, we can use (1) of Lemma \ref{lem:lower_free_rank_preparation} to decompose each $\sigma_j$ as $\sigma_j'\phi_{a,b}$, where $\sigma'_j$ is the restriction of $\sigma_j$ to $(\cA\setminus\{b\})^+$. Since $\phi_{a,b}$ is letter-onto and $\ell_c = |\sigma'_j(c)|$ for every $j \in J$, $c\in\cA\setminus\{b\}$, the conclusion of the lemma holds, contrary to our assumptions.

It left to consider the case in which $\ell_a \not= \ell_b$. We only do the case $\ell_a < \ell_b$ as the other is similar. Then, by \eqref{eq:lower_free_rank_s_equal_0:2}, for every $j \in J$ there exists a non-empty word $t_j \in \cB_j^{\ell_b-\ell_a}$ of length $\ell_b-\ell_a$ such that $\sigma_j(b) = \sigma_j(a)t_j$. Thus, we can use (2) of Lemma \ref{lem:lower_free_rank_preparation} to write, for any $j \in J$, $\sigma_j = \sigma'_j\psi_{a,b}$, where $\sigma'_j$ is defined as in \eqref{eq:lem:lower_free_rank_preparation:cut}. 

Let $\tilde{u} = \psi_{a,b}(u')$ and $\tilde{v} = b\psi_{a,b}(v')$. We want now to prove that $\tilde{u}$, $\tilde{v}$ and $\{\sigma'_j : j \in J\}$ satisfy the hypothesis of the lemma. First, we observe that for every $j \in J$,
\begin{equation}\label{eq:lower_free_rank_s_equal_0:3}
\text{if $c\not=b$, then $|\sigma'_j(c)| = \ell_c$, and $|\sigma'_j(b)| = |t_j| = \ell_b-\ell_a$.}
\end{equation}
Therefore, $\{\sigma'_j\}_{j\in J}$ satisfy condition \eqref{eq:lower_free_rank_s_equal_0:1}. Also, since $\psi_{a,b}(c)$ never starts with $b$, we have that
\begin{equation}\label{eq:lower_free_rank_s_equal_0:4}
\text{$\tilde{u}$, $\tilde{v}$ start with different letters.}
\end{equation}
Furthermore, by using the symbol $\leq_p$ to denote the prefix relation, we can compute:
\begin{align*}
\sigma_j(a)\sigma'_j(\tilde{u}) &= \sigma_j(a)\sigma_j(u') = \sigma_j(u) 
\leq_p 
\sigma_j(v) = \sigma'_j(\psi_{a,b}(v)) = \sigma'_j(a)\sigma'_j(\tilde{v}).
\end{align*}
This and the fact that $\sigma_j(a)$ is equal to $\sigma'_j(a)$ imply that
\begin{equation}\label{eq:lower_free_rank_s_equal_0:5}
\text{$\sigma'_j(\tilde{u})$ is a prefix of $\sigma'_j(\tilde{v})$ for every $j \in J$.}
\end{equation}
Finally, we note
\begin{equation}\label{eq:lower_free_rank_s_equal_0:6}
|\tilde{u}| \geq |u|-1 \geq \sum_{c\in\cA}\ell_c - \ell_a \eqqcolon \ell'.
\end{equation}
We conclude from equations \eqref{eq:lower_free_rank_s_equal_0:3}, \eqref{eq:lower_free_rank_s_equal_0:4}, \eqref{eq:lower_free_rank_s_equal_0:5} and \eqref{eq:lower_free_rank_s_equal_0:6} that $\tilde{u}$, $\tilde{v}$ and $\{\sigma'_j : j \in J\}$ satisfy the hypothesis of this lemma. Since $\ell' < \ell$, the minimality of $\ell$ implies that there exist a letter-onto morphism $q'\colon\cA^+\to\cC^+$, with $\#\cC < \#\cA$, and morphisms $\{p_j\colon\cC^+\to\cB_j^+\}_{j\in J}$ satisfying $\sigma'_j=p_jq'$ and a property analogous to \eqref{eq:lower_free_rank_s_equal_0:1}. But then $q \coloneqq q'\psi_{a,b}$ is also letter-onto and the morphisms $\{p_j\}_{j\in J}$ satisfy $\sigma_j = p_jq$ and a property analogous to \eqref{eq:lower_free_rank_s_equal_0:1}. Thus, the conclusion of the lemma holds for $\{\sigma_j\}_{j\in J}$, contrary our assumptions.
\end{proof}

\begin{lem}\label{lem:lower_free_rank_s_NOTequal_0}
Let $\sigma\colon\cA^+\to\cB^+$ be a morphism, $u,v \in \cA^+$, $a,b$ be the first letters of $u,v$, respectively, and $\sigma(a) = st$ be a decomposition of $\sigma(a)$ in which $t$ is non-empty. Assume that $\sigma(u)$ is a prefix of $s\sigma(v)$, $|u| \geq |\sigma|_1+|s|$, and either that $s = 1$ and $a\not=b$ or that $s \not= 1$.

Then, there exist morphisms $q\colon\cA^+ \to \cC^+$ and $p\colon\cC^+\to\cB^+$ such that $\#\cC \leq \#\cA$, $q$ is letter-onto, $|p|_1 < |\sigma|_1$, and $\sigma = pq$.
\end{lem}

\begin{rem}\label{rem:lem:lower_free_rank_s_NOTequal_0}
As in Lemma \ref{lem:lower_free_rank_s_equal_0}, there are symmetric hypothesis for the previous lemma that involve suffixes instead of prefixes and which give the same conclusion. We will use this in the proof of Lemma \ref{lem:factorize_sequence}.
\end{rem}
\begin{proof}[Proof (of Lemma \ref{lem:lower_free_rank_s_NOTequal_0})]
Let us write $u = au'$ and $v = bv'$. We first consider the case in which $s=1$. In this situation, $u$ and $v$ start with different letters, so Lemma \ref{lem:lower_free_rank_s_equal_0} can be applied (with the index set $J$ chosen as a singleton) to obtain a decomposition $\cA^+ \overset{q}{\to} \cC^+ \overset{p}{\to} \cB^+$ such that $q$ is letter-onto, $\#\cC < \#\cA$, and $\sigma = pq$. Since $\cC$ has strictly fewer elements than $\cA$, we have $|p|_1 < |\sigma|_1$. Hence, the conclusion of the lemma holds in this case.

We now assume that $s\not=1$. In this case, $t$ and $s$ are non-empty, so we can use (3) of Lemma \ref{lem:lower_free_rank_preparation} to factorize $\sigma = \sigma'\theta_{a, \tilde{a}}$, where $\tilde{a}$ is a letter not in $\cA$ and $\sigma'$ is defined as in \eqref{eq:lem:lower_free_rank_preparation:split}. We set $\tilde{u} = a\theta_{a,\tilde{a}}(u')$ and $\tilde{v} = \theta_{a,\tilde{a}}(v)$. Our plan is to use Lemma \ref{lem:lower_free_rank_s_equal_0} with $\tilde{u}$, $\tilde{v}$ and $\sigma'$.

Observe that $\theta_{a,\tilde{a}}(c)$ never starts with $a$, so
\begin{equation}\label{eq:lower_free_rank_s_NOTequal_0:1}
\text{$\tilde{u}$, $\tilde{v}$ start with different letters.}
\end{equation}
Also, by using, as in the previous proof, the symbol $\leq_p$ to denote the prefix relation, we can write:
\begin{equation*}
s\sigma'(\tilde{u}) =
s\sigma'(a) \sigma'(\theta_{a,\tilde{a}}(u')) =
st\sigma(u') = \sigma(u) \leq_p
s\sigma(v) = s\sigma'(\theta_{a,\tilde{a}}(v)) =
s\sigma'(\tilde{v}),
\end{equation*}
which implies that
\begin{equation}\label{eq:lower_free_rank_s_NOTequal_0:2}
\text{$\sigma'(\tilde{u})$ is a prefix of $\sigma'(\tilde{v})$.}
\end{equation}
Finally, we use \eqref{eq:lem:lower_free_rank_preparation:split} to compute:
\begin{equation}\label{eq:lower_free_rank_s_NOTequal_0:3}
|\tilde{u}| \geq |u|-1 \geq 
|\sigma|_1 +|s|-1 \geq |\sigma|_1 = |\sigma'|_1.
\end{equation}
We conclude, by equations \eqref{eq:lower_free_rank_s_NOTequal_0:1}, \eqref{eq:lower_free_rank_s_NOTequal_0:2} and \eqref{eq:lower_free_rank_s_NOTequal_0:3}, that Lemma \ref{lem:lower_free_rank_s_equal_0} can be applied with $\tilde{u}$, $\tilde{v}$ and $\sigma'$ (and $J$ as a singleton). Thus, there exist morphisms $q'\colon(\cA\cup\{\tilde{a}\})^+\to\cC^+$ and $p\colon\cC^+\to\cB^+$ such that $\#\cC < \#(\cA\cup\{\tilde{a}\})$, $q'$ is letter-onto and $\sigma'=pq'$. Then, $\#\cC \leq \#\cA$, $q \coloneqq q'\theta_{a, \tilde{a}}$ is letter-onto and $\sigma = p q'\theta_{a, \tilde{a}} = pq$. Moreover, since $\theta_{a, \tilde{a}}$ is not the identity function, we have $|p|_1 < |\sigma|_1$.
\end{proof}

The next lemma is the main result of this subsection. To state it, we introduce additional notation. For an alphabet $\cA$, let $\cA^{++}$ be the set of words $w \in \cA^+$ in which all letters occur. Observe that $\sigma\colon\cA^+\to\cB^+$ is letter-onto if and only if $\sigma(\cA^{++}) \subseteq \cB^{++}$.
\begin{lem}\label{lem:factorize_sequence}
Let $\phi\colon \cA^+ \to \cC^+$, $\tau\colon \cB^+\to\cC^+$ be morphisms such that $\tau$ is $\ell$-proper, with $\ell \geq |\phi|_1^4$, and $\phi(\cA^+) \cap \tau(\cB^{++}) \not= \emptyset$. Then, there exist $\cB^+ \overset{q}{\longrightarrow} \cD^+ \overset{p}{\longrightarrow} \cC^+$ such that
\begin{equation*}
\text{(i) $\#\cD \leq \#\cA$,}\quad
\text{(ii) $\tau = pq$,}\quad
\text{(iii) $q$ is letter-onto and proper.}
\end{equation*}
\end{lem}
\begin{proof}
By contradiction, we suppose that the lemma does not hold for $\phi$ and $\tau$ and, moreover, that $|\phi|_1$ as small as possible.

That $\phi(\cA)^+ \cap \tau(\cB^{++})$ is non-empty means that there exist $u = u_1\cdots u_n \in \cA^+$ and $w = w_1\cdots w_m \in \cB^{++}$ with $\phi(u) = \tau(w)$. If $m = 1$, then, since $w \in \cB^{++}$, we have $\#\cB = \{v_1\}$ and the conclusion of the lemma trivially holds for $\cD = \{a\in\cC:\text{$a$ occurs in $\tau(w_1)$}\}$, $q\colon\cB^+\to\cD^+$, $w_1\mapsto\tau(w_1)$, and $p\colon\cD^+\to\cC^+$ the inclusion map, contradicting our initial assumption. Therefore, $m \geq 2$ and $\{1,\dots,m-1\}$ is non-empty. 

Let $k \in \{1,\dots,m-1\}$. We define $i_k$ as the smallest number in $\{1,\dots,n\}$ for which $|\tau(w_1\cdots w_k)| < |\phi(u_1\cdots u_{i_k})|$ holds. Since $|\phi(u_1)| \leq |\phi|_1 \leq \ell \leq |\tau(w_1\cdots w_k)|$, $i_k$ is at least $2$ and, thus, $|\phi(u_1\cdots u_{i_k-1})| \leq |\tau(w_1\cdots w_k)|$ by minimality of $i_k$. Hence, there exists a decomposition $\phi(u_{i_k}) = s_kt_k$ such that $t_k$ is non-empty and
\begin{equation}\label{eq:lem:factorize_sequence:1}
t_k\phi(u_{i_k+1}\dots u_n) = \tau(w_{k+1}\dots w_m).
\end{equation}

Our next objective is to use Lemma \ref{lem:lower_free_rank_s_NOTequal_0} to prove that $s_k$ and $u_k$ have a very particular form:
\begin{claim}
For every $k\in\{1,\dots,m-1\}$, $s_k = 1$ and $u_1 = u_{i_k}$.
\end{claim}
\begin{claimproof}
To prove this, we suppose that it is not true, this is, that there exists $k\in\{1,\dots,m-1\}$ such that
\begin{equation}\label{eq:lem:factorize_sequence:2}
\text{$s_k \not= 1$ or $u_1 \not= u_{i_k}$.}
\end{equation}
Let $\tilde{u} \coloneqq u_{i_k}\dots u_{i_k+|\phi|^2_1-1}$ and $\tilde{v} \coloneqq u_1\dots u_{|\phi|^3_1}$. We are going to check the hypothesis of Lemma \ref{lem:lower_free_rank_s_NOTequal_0} for $\tilde{u}$, $\tilde{v}$ and $\phi$.

First, we observe that, since $\phi(u)=\tau(v)$, we have that $\phi(\tilde{v})$ is a prefix of $\tau(v)$. Moreover, given that $|\phi(\tilde{v})| \leq |\phi|_1^4 \leq \ell$ and that $\tau$ is $\ell$-proper, $\phi(\tilde{v})$ is a prefix of $\tau(b)$ for every $b \in \cB$. In particular,
\begin{equation}\label{eq:lem:factorize_sequence:3}
\text{$\phi(\tilde{v})$ is a prefix of $\tau(w_k)$}.
\end{equation}
Second, from \eqref{eq:lem:factorize_sequence:1} and the inequalities $|t_k\phi(u_{i_k+1}\dots u_{i_k+|\phi|^2_1-1})| \leq |\phi|_1^3 \leq \ell \leq |\tau(w_k)|$ we deduce that $t_k\phi(u_{i_k+1}\dots u_{i_k+|\phi|^2_1-1})$ is a prefix of $\tau(w_k)$.
Therefore,
\begin{equation}\label{eq:lem:factorize_sequence:4}
\text{$\phi(\tilde{u}) = s_kt_k\phi(u_{i_k+1}\dots u_{i_k+|\phi|^2_1-1})$ is a prefix of $s_k\tau(w_k)$.}
\end{equation}
We conclude from \eqref{eq:lem:factorize_sequence:3}, \eqref{eq:lem:factorize_sequence:4} and the inequality $|\phi(\tilde{u})| \leq |\phi|_1^3 = |\tilde{v}| \leq |s_k\phi(\tilde{v})|$ that
\begin{equation*}
\text{$\phi(\tilde{u})$ is a prefix of $s_k\phi(\tilde{v})$.}
\end{equation*}
This, the inequality $|\tilde{u}| \geq |\phi|_1+|s_k|$ and \eqref{eq:lem:factorize_sequence:2} allow us to use Lemma \ref{lem:lower_free_rank_s_NOTequal_0} with $\tilde{u}$, $\tilde{v}$ and $\phi$ and obtain morphisms $\cA^+ \overset{\tilde{q}}{\longrightarrow} \tilde{\cA}^+ \overset{\tilde{\phi}}{\longrightarrow} \cC^+$ such that $\#\tilde{\cA} \leq \#\cA$, $\phi = \tilde{\phi}\tilde{q}$ and $|\tilde{\phi}|_1 < |\phi|_1$. Then, $\ell \geq |\phi|_1^4 > |\tilde{\phi}|_1^4$ and $\tilde{\phi}(\tilde{\cA}^+) \cap \tau(\cB^{++})$ contains the element $\tilde{\phi}(\tilde{q}(u)) = \tau(w)$, and so $\tau$ and $\tilde{\phi}$ satisfy the hypothesis of this lemma. Therefore, by the minimality of $|\phi|_1$, there exists a decomposition $\cB^+ \overset{q}{\to} \cD^+ \overset{p}{\to} \cC^+$ of $\tau$ satisfying (i-iii) of this lemma, contrary to our assumptions.
\end{claimproof}

An argument similar to the one used in the proof of the previous claim gives us that
\begin{equation}\label{eq:lem:factorize_sequence:5}
\text{$u_n = u_{i_k-1}$ for every $k\in\{1,\dots,m-1\}$.}
\end{equation}
We refer the reader to Remark \ref{rem:lem:lower_free_rank_s_NOTequal_0} for further details. 

Now we can finish the proof. First, from \eqref{eq:lem:factorize_sequence:1} and the first part of the claim we get that $\tau(w_k) = \phi(u_{i_{k-1}}\cdots u_{i_k-1})$ for $k \in \{2,\dots,m-1\}$, $\tau(w_1) = \phi(u_1\cdots u_{i_1-1})$ and $\tau(w_m) = \phi(u_{i_{m-1}}\cdots u_n)$. Being $w \in \cB^{++}$, these equations imply that each $\tau(b)$, $b \in \cB$, can be written as a concatenation $x_1\cdots x_N$, with $x_j \in \phi(\cA)$. Moreover, by the second part of the claim and \eqref{eq:lem:factorize_sequence:5}, we can choose this decomposition so that $x_1 = u_1$ and $x_N = u_n$. This defines (maybe non-unique) morphisms $\cB^+ \overset{q}{\longrightarrow} \cD_1^+ \overset{p_1}{\longrightarrow} \cC^+$ such that $\tau = p_1q$, $\#\cD_1 \leq \#\{\phi(u_1),\dots,\phi(u_n)\} \leq \#\cA$ and $q$ is proper. If we define $\cD$ as the set of letters $d \in \cD_1$ that occur in some $w \in q(\cB)$, and $p$ as the restriction of $p_1$ to $\cD$, then we obtain a decomposition $\cB^+ \overset{q}{\longrightarrow} \cD^+ \overset{p}{\longrightarrow} \cC^+$ that still satisfies the previous properties, but in which $q$ is letter-onto. Hence, $p$ and $q$ met conditions (i), (ii) and (iii).
\end{proof}

\subsection{Periodicity lemmas}

We will also need classic results from combinatorics on words. We follow the presentation of \cite[Chapter 6]{handbook_words}.

Let $w\in \cA^*$ be a non-empty word. We say that $p$ is a {\em local period} of $w$ at the position $|u|$ if $w=uv$, with $u,v\not=1$, and there exists a word $z$, with $|z|=p$, such that one of the following conditions holds for some words $u'$ and $v'$:
\begin{align}\label{eq:defi:local_period}
\begin{cases}
	(i)\quad & u=u'z \ \text{and}\ v=zv';\\
	(ii)\quad & z =u'u \ \text{and}\ v=zv';\\
	(iii)\quad &u =u'z \ \text{and}\ z=vv';\\
	(iv)\quad & z =u'u = vv'.\end{cases}
\end{align}
Further, the {\em local period of $w$ at the position $|u|$}, in symbols $\per(w,u)$, is defined as the smallest local period of $w$ at the position $u$. It follows directly from \eqref{eq:defi:local_period} that $\per(w,u) \leq \per(w)$.

\begin{figure}[H]
\makebox[\textwidth][c]{\includegraphics[scale=0.35]{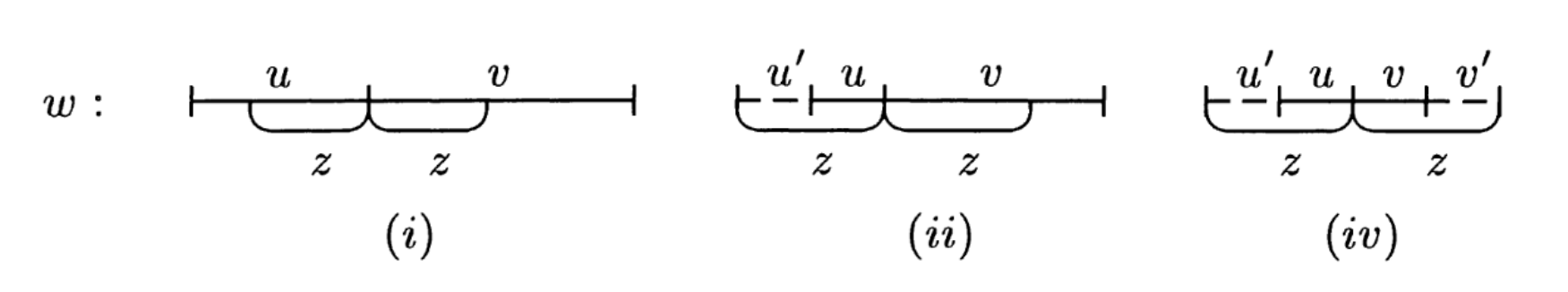}}
\caption{The illustration of a local period.}
\label{fig:local_period}
\end{figure}

\begin{theo}[Critical Factorization Theorem]\label{theo:factorizacion_critica}
Each non-empty word $w \in \cA^*$, with $|w|\geq2$, possesses at least one factorization $w = uv$, with $u,v\not=1$, which is {\em critical}, {\em i.e.}, $\per(w) = \per(w,u)$.
\end{theo}

\section{Rank of symbolic factors}\label{sec:rank_of_factors}

In this section we prove Theorem \ref{theo:intro:1}. We start by introducing the concept of {\em factor} between directive sequences and, in Proposition \ref{prop:make_sequence_reco}, its relation with factor maps between $\cS$-adic subshifts. These ideas are the $\cS$-adic analogs of the concept of {\em premorphism} between ordered Bratteli diagrams from \cite{AEG2015} and their Proposition 4.6. Although Proposition \ref{prop:make_sequence_reco} can be deduced from Proposition 4.6 in \cite{AEG2015} by passing from directive sequences to ordered Bratteli diagrams and backwards, we consider this a little bit artificial since it is possible to provide a direct combinatorial proof; this is done in the Appendix. It is interesting to note that our proof is constructive (in contrast of the existential proof in \cite{AEG2015}) and shows some additional features that are consequence of the combinatorics on words analysis made. 

Next, we use ideas from \cite{E20} and \cite{GH2020} to prove Theorem \ref{theo:intro:1}. In particular, this improves the previous bounds in \cite{E20} and \cite{GH2020} to the best possible one. We apply these results, in Corollary \ref{cor:question_1}, to answer affirmatively Question \ref{ques:is_factor_FTR} and, in Theorem \ref{theo:coalescency_chain_factors}, to prove a strong coalescence property for the class of systems considered in Theorem \ref{theo:intro:1}. It is worth noting that this last result is only possible due the bound in Theorem \ref{theo:intro:1} being optimal. We end this section by proving that Cantor factors of finite topological rank systems are either subshifts of odometers.

\subsection{Rank of factors of directive sequences}

The following is the $\cS$-adic analog of the notion of {\em premorphism} between ordered Bratteli diagrams in \cite{AEG2015}.
\begin{defi}\label{defi:homomorphism}
Let $\boldsymbol{\sigma} = (\cA_{n+1}^+\to\cA_n^+)_{n\in\N}$, $\boldsymbol{\tau} = (\cB_{n+1}^+\to\cB_n^+)_{n\in\N}$ be directive sequences. A {\em factor} $\boldsymbol{\phi}\colon \boldsymbol{\sigma}\to \boldsymbol{\tau}$ is a sequence of morphisms $\boldsymbol{\phi} = (\phi_n)_{n\in\N}$, where $\phi_0\colon\cA_1^+\to\cB_0^+$ and $\phi_n\colon\cA_n^+\to\cB_n^+$ for $n \geq1$, such that $\phi_0 = \tau_0\phi_1$ and $\phi_n\sigma_n = \tau_n\phi_{n+1}$ and for every $n \geq 1$.

We say that $\boldsymbol{\phi}$ is {\em proper} (resp.\ {\em letter-onto}) if $\phi_n$ is proper (resp.\ {\em letter-onto}) for every $n\in\N$.
\end{defi}

\begin{rem}\label{rem:homomorphism} Factors are not affected by contractions. More precisely, if $0=n_0<n_1<n_2<\dots$, then $\boldsymbol{\phi'} = (\phi_{n_k})_{k\in\N}$ is a factor from $\boldsymbol{\sigma'} = (\sigma_{[n_k,n_{k+1})})_{k\in\N}$ to $\boldsymbol{\tau'} = (\tau_{[n_k,n_{k+1})})_{k\in\N}$.
\end{rem}

The next lemma will be needed at the end of this section.
\begin{lem}\label{lem:regularity_factor_Ssequences}
Let $\boldsymbol{\phi} = (\phi_n)_{n\geq1}\colon\boldsymbol{\sigma}\to\boldsymbol{\tau}$ be a factor. Assume that $\boldsymbol{\sigma}$ and $\boldsymbol{\tau}$ are everywhere growing and proper and that $\boldsymbol{\phi}$ is letter-onto. Then, $X_{\boldsymbol{\tau}} = \bigcup_{k\in\Z}T^k\phi_0(X_{\boldsymbol{\sigma}}^{(1)})$ and $X_{\boldsymbol{\tau}}^{(n)} = \bigcup_{k\in\Z}T^k\phi_n(X_{\boldsymbol{\sigma}}^{(n)})$ for every $n \geq 1$.
\end{lem}
\begin{proof}
We start by proving that $X_{\boldsymbol{\tau}}^{(n)} \subseteq \bigcup_{k\in\Z}T^k\phi_n(X_{\boldsymbol{\sigma}}^{(n)})$. Let $y \in X_{\boldsymbol{\tau}}^{(n)}$ and $\ell \in \N$. There exist $N>n$ and $b \in \cB_n$ such that $y_{[-\ell,\ell]}$ occurs in $\tau_{[n,N)}(b)$. In addition, since $\phi_N$ is letter-onto, there exists $a \in \cA_N$ for which $b$ occurs in $\phi_N(a)$. Then, $y_{[-\ell,\ell]}$ occurs in $\tau_{[n,N)}\phi_N(b)$ and, consequently, also in $\phi_n\sigma_{[n,N)}(b)$ as $\tau_{[n,N)}\phi_N = \phi_n\sigma_{[n,N)}$. Hence, by taking the limit $\ell\to\infty$ we can find $(k',x) \in \Z\times X_{\boldsymbol{\sigma}}^{(n)}$ such that $y = T^{k'}\phi_n(x)$. Therefore, $y \in \bigcup_{k\in\Z}T^k\phi_n(X_{\boldsymbol{\sigma}}^{(n)})$. 
To prove the other inclusion, we use Lemma \ref{lem:eg_and_proper_imply_limitpoints} to compute:
\begin{align*}
\phi_n(X_{\boldsymbol{\sigma}}^{(n)}) &=
\bigcap_{N>n}\bigcup_{k\in\Z}T^k\phi_n\sigma_{[n,N)}(\cA_N^\Z) =
\bigcap_{N>n}\bigcup_{k\in\Z}T^k\tau_{[n,N)}\phi_N(\cA_N^\Z) 
\\&\subseteq
\bigcap_{N>n}\bigcup_{k\in\Z}T^k\tau_{[n,N)}(\cB_N^\Z) =
X_{\boldsymbol{\tau}}^{(n)}.
\end{align*}
\end{proof}

As we mentioned before, the following proposition is consequence of the main result in \cite{AEG2015}. We provide a combinatorial proof in the Appendix.
\begin{prop}\label{prop:make_sequence_reco}
Let $\boldsymbol{\sigma}$ be a letter-onto, everywhere growing and proper directive sequence. Suppose that $X_{\boldsymbol{\sigma}}$ is aperiodic. Then, there exist a contraction $\boldsymbol{\sigma}' = (\sigma'_n)_{n\in\N}$, a letter-onto, everywhere growing, proper and recognizable $\boldsymbol{\tau} = (\tau_n)_{n\in\N}$ generating $X_{\boldsymbol{\sigma}}$, and a letter-onto factor $\boldsymbol{\phi}\colon \boldsymbol{\sigma}' \to \boldsymbol{\tau}$, $\boldsymbol{\phi} = (\phi_n)_{n\in\N}$, such that $\phi_0 = \sigma'_0$.
\end{prop}

The next proposition is the main technical result of this section. To state it, it is convenient to introduce the following concept. The directive sequences $\boldsymbol{\sigma}$ and $\boldsymbol{\tau}$ are {\em equivalent} if $\boldsymbol{\sigma} = \boldsymbol{\nu}'$, $\boldsymbol{\tau} = \boldsymbol{\nu}''$ for some contractions $\boldsymbol{\nu}'$, $\boldsymbol{\nu}''$ of a directive sequence $\boldsymbol{\nu}$. Observe that equivalent directive sequences generate the same $\cS$-adic subshift.
\begin{prop}\label{prop:decrese_rank_factor}
Let $\boldsymbol{\phi}\colon \boldsymbol{\sigma} \to \boldsymbol{\tau}$ be a letter-onto factor between the everywhere growing and proper directive sequences. Then, there exist a letter-onto and proper factor $\boldsymbol{\psi}\colon \boldsymbol{\sigma}'\to\boldsymbol{\nu}$, where
\begin{enumerate}
\item $\boldsymbol{\sigma}'$ is a contraction of $\boldsymbol{\sigma}$;
\item $\boldsymbol{\nu}$ is letter-onto, everywhere growing, proper, equivalent to $\boldsymbol{\tau}$, $\mathrm{AR}(\boldsymbol{\nu}) \leq \mathrm{AR}(\boldsymbol{\sigma})$, and the first coordinate of $\boldsymbol{\psi}$ and $\boldsymbol{\phi}$ coincide;
\item if $\boldsymbol{\tau}$ is recognizable, then $\boldsymbol{\nu}$ is recognizable.
\end{enumerate}
\end{prop}
\begin{proof}
Let us write $\boldsymbol{\sigma} = (\cA_{n+1}^+\to\cA_n^+)_{n\in\N}$ and $\boldsymbol{\tau} = (\cB_{n+1}^+\to\cB_n^+)_{n\in\N}$. Up to contractions, we can suppose that for every $n \geq1$, $\#\cA_n = \mathrm{AR}(\boldsymbol{\sigma})$ and that $\tau_n$ is $|\phi_n|^4_1$-proper (for the last property we used that $\boldsymbol{\tau}$ is everywhere growing and proper). 

Using that $\phi_{n+1}$ is letter-onto, we can compute:
\begin{equation*}
\tau_n(\cB_{n+1}^{++})
\supseteq \tau_n(\phi_{n+1}(\cA_{n+1}^{++})) =
\phi_n(\sigma_n(\cA_{n+1}^{++})) \subseteq \phi_n(\cA_n^+),
\end{equation*}
where in the middle step we used the commutativity property of $\boldsymbol{\phi}$. We deduce that
\begin{equation*}
\text{$\tau_n(\cB_{n+1}^{++}) \cap \phi_n(\cA_n^+) \not= \emptyset$ for every $n \in \N$.}
\end{equation*}
This and the fact that $\tau_n$ is a $|\phi_n|^4_1$-proper morphism allow us to use Lemma \ref{lem:factorize_sequence} to find morphisms $\cB_{n+1}^+ \overset{q_{n+1}}{\longrightarrow} \cD_{n+1}^+ \overset{p_n}{\longrightarrow} \cB_n^+$ such that 
\begin{equation*}
\text{(i) $\#\cD_{n+1} \leq \#\cA_n$, (ii) $\tau_n = p_nq_{n+1}$, (iii) $q_{n+1}$ is letter-onto and proper.}
\end{equation*}
We define $\nu_0 \coloneqq p_0$, the morphisms $\nu_n \coloneqq q_np_n \colon \cD_{n+1}^+ \to \cD_n^+$ and $\psi_n \coloneqq q_n\phi_n\colon\cA_n^+\to\cD_n^+$, $n \geq1$, and the sequences $\boldsymbol{\nu} = (\nu_n)_{n\in\N}$ and $\boldsymbol{\psi} = (\psi_n)_{n\in\N}$, where $\psi_0 \coloneqq \phi_0$. We are going to show that these objects satisfy the conclusion of the proposition.

We start by observing that it follows from the definitions that the diagram below commutes for all $n\geq1$:
\begin{equation*}
\begin{tikzcd}
\cA_n^+ \arrow{r}{\phi_n}
& \cB_n^+ \arrow{r}{q_n}
& \cD_n^+
\\
\cA_{n+1}^+ \arrow{u}{\sigma_n} \arrow{r}[swap]{\phi_{n+1}}
& \cB_{n+1}^+ \arrow{u}{\tau_n} \arrow{r}[swap]{q_{n+1}}
& \cD_{n+1}^+ \arrow{ul}{p_n} \arrow{u}[swap]{\nu_n}
\end{tikzcd}
\end{equation*}
In particular, $\nu_n\nu_{n+1} = q_n\tau_np_{n+1}$, so $\langle \nu_{[n,n+1]} \rangle \geq \langle \tau_n\rangle$. Being $\boldsymbol{\tau}$ everywhere growing, this implies that $\boldsymbol{\nu}$ has the same property. We also observe that (iii) implies that $\nu_n = q_np_n$ is letter-onto and proper. Altogether, these arguments prove that, up to contracting the first levels, $\boldsymbol{\nu}$ is everywhere growing and proper. 

Next, we note that $\boldsymbol{\nu}$ and $\boldsymbol{\tau}$ are equivalent as both are contractions of $(p_0, q_1, p_1, q_2, \dots)$. This implies, by Lemma \ref{lem:recognizability_of_composition}, that $\boldsymbol{\nu}$ is recognizable if $\boldsymbol{\tau}$ is recognizable. Further, by (i), $\boldsymbol{\nu}$ has alphabet rank at most $\mathrm{AR}(\boldsymbol{\sigma})$. 

It only left to prove that $\boldsymbol{\psi}$ is a letter-onto and proper factor. By unraveling the definitions we can compute:
\begin{equation*}
\psi_0 = \phi_0 = \tau_0\phi_1 = p_0q_1\phi_1 = \nu_0\psi_1,
\end{equation*}
and from the diagram we have $\sigma_n\psi_n = \psi_{n+1}\tau_n$ for all $n\geq1$. Therefore, $\boldsymbol{\psi}$ is a factor. Finally, since $q_n$ is letter-onto and proper by (iii) and $\boldsymbol{\phi}$ was assumed to be letter-onto, $\psi_n = q_n\phi_n$ is letter-onto and proper.
\end{proof}


\subsection{Rank of factors of $\cS$-adic subshifts}
In this section we will prove Theorem \ref{theo:intro:1} and its consequences. We start with a technical lemma.

The next lemma will allow us to assume without loss of generality that our directive sequences are letter-onto.
\begin{lem}\label{lem:non_letter-onto_to_letter-onto}
Let $\boldsymbol{\tau} = (\tau_n\colon\cA_{n+1}^+\to\cA_n^+)_{n\in\N}$ be an everywhere growing and proper directive sequence. If $\tilde{\cA}_n = \cA_n \cap \cL(X_{\boldsymbol{\sigma}}^{(n)})$, $\tilde{\tau}_n$ is the restriction of $\tau_n$ to $\tilde{\cA}_{n+1}$ and $\boldsymbol{\tilde{\tau}} = (\tilde{\tau}_0,\tilde{\tau}_1,\dots)$, then $\boldsymbol{\tilde{\tau}}$ is letter-onto and $X_{\boldsymbol{\tilde{\tau}}}^{(n)} = X_{\boldsymbol{\tau}}^{(n)}$ for every $n \in \N$. Conversely, if $\boldsymbol{\tau}$ is letter-onto, then $\cA_n \subseteq \cL(X_{\boldsymbol{\tau}}^{(n)})$ for every $n \in \N$.
\end{lem}
\begin{proof}
By Lemma \ref{lem:desubstitution}, $\tilde{\tau}_n$ is letter-onto mapping $\tilde{\cA}_{n+1}^+$ into $\tilde{\cA}_n$. Moreover, that lemma also gives that for every $x \in X_{\boldsymbol{\tau}}^{(n)}$ and $N>n$, there exists a $\tau_{[n,N)}$-factorization $(k',x')$ of $x$ in $X_{\boldsymbol{\tau}}^{(N)}$. This together with the inclusion $X_{\boldsymbol{\tau}}^{(N)} \subseteq \tilde{\cA}_N^\Z$ imply that
\begin{equation*}
Z \coloneqq 
\bigcap_{N>n}\bigcup_{k\in\Z}T^k\tau_{[n,N)}(\tilde{\cA}_N^\Z) \supseteq
X_{\boldsymbol{\tau}}^{(n)}
\end{equation*}
Now, $\boldsymbol{\tilde{\tau}}$ is everywhere growing and proper, so we can apply Lemma \ref{lem:eg_and_proper_imply_limitpoints} to obtain that $X_{\boldsymbol{\tilde{\tau}}}^{(n)} = Z \supseteq X_{\boldsymbol{\tau}}^{(n)}$. Since it is clear that $X_{\boldsymbol{\tilde{\tau}}}^{(n)}\subseteq X_{\boldsymbol{\tau}}^{(n)}$ as $\tilde{\cA}_N \subseteq \cA_N$ for every $N$, we conclude that $X_{\boldsymbol{\tilde{\tau}}}^{(n)} = X_{\boldsymbol{\tau}}^{(n)}$.

If $\boldsymbol{\tau}$ is letter-onto, then $\cA_n \subseteq \cL(\bigcup_{k\in\Z}T^k\tau_{[n,N)}(\cA_N^\Z))$ for every $N>n$, and hence, by the formula in Lemma \ref{lem:eg_and_proper_imply_limitpoints}, $\cA_n \subseteq \cL(X_{\boldsymbol{\tau}}^{(n)})$.
\end{proof}

Now we are ready to prove Theorem \ref{theo:intro:1}. We re-state it in a more precise way.
\begin{ctheo}{\ref{theo:intro:1}}\label{theo:rank_of_factor}
Let $\pi\colon(X,T) \to(Y,T)$ be a factor map between aperiodic subshifts. Suppose that $X$ is generated by the everywhere growing and proper directive sequence $\boldsymbol{\sigma} = (\sigma_n\colon\cA_{n+1}^+\to\cA_n^+)_{n\in\N}$ of alphabet rank $K$. Then, $Y$ is generated by a letter-onto, everywhere growing, proper and recognizable directive sequence $\boldsymbol{\tau}$ of alphabet rank at most $K$.

Moreover, if $\boldsymbol{\sigma}$ is letter-onto, then, up to contracting the sequences, there exists a proper factor $\boldsymbol{\phi}\colon\boldsymbol{\sigma}\to\boldsymbol{\tau}$ such that $\pi(\sigma_0(x)) = \phi_0(x)$ for all $x \in X_{\boldsymbol{\sigma}}^{(1)}$ and $|\sigma_0(a)| = |\phi_0(a)|$ for all $a \in \cA_{1}$
\end{ctheo}
\begin{proof}
Thanks to Lemma \ref{lem:non_letter-onto_to_letter-onto}, we can assume without loss of generality that $\boldsymbol{\sigma}$ is letter-onto. Moreover, in this case we have: 
\begin{equation}\label{eq:theo:rank_of_factor:0}
\text{$\cA_n \subseteq \cL(X_{\boldsymbol{\sigma}}^{(n)})$ for every $n \in \N$.}
\end{equation}

Let us write $\boldsymbol{\sigma} = (\sigma_n\colon\cA_{n+1}^+\to\cA_n^+)_{n\in\N}$. By contracting $\boldsymbol{\sigma}$, we can further assume that $\sigma_0$ is $r$-proper and $\pi$ has radius $r$. Then, Lemma \ref{lem:factor_to_morphism} gives us a proper morphism $\tau\colon\cA_1^+\to\cB^+$, where $\cB$ is the alphabet of $Y$, such that 
\begin{equation}\label{eq:theo:rank_of_factor:1}
\text{$\pi(\sigma_0(x)) = \tau(x)$ for all $x \in X_{\boldsymbol{\sigma}}^{(1)}$ and $|\sigma_0(a)| = |\tau(a)|$ for every $a \in \cA_1$.}
\end{equation}
In particular, $\pi(\sigma_{[0,n)}(x)) = \tau\sigma_{[1,n)}(x)$ and $|\sigma_{[0,n)}(a)| = |\tau\sigma_{[1,n)}(a)|$ for all $n \in \N$, $x \in X_{\boldsymbol{\sigma}}^{(n)}$ and $a \in \cA_n$, so \eqref{eq:theo:rank_of_factor:1} holds for any contraction of $\boldsymbol{\sigma}$.

We define $\boldsymbol{\tilde{\sigma}} = (\tau, \sigma_1, \sigma_2, \dots)$ and observe this is a letter-onto, everywhere growing and proper sequence generating $Y$. This and that $Y$ is aperiodic allow us to use Proposition \ref{prop:make_sequence_reco} and obtain, after a contraction, a letter-onto factor $\boldsymbol{\tilde{\phi}}\colon \boldsymbol{\tilde{\sigma}}\to \boldsymbol{\tilde{\tau}}$, where $\tilde{\phi}_0 = \tilde{\sigma}_0 = \tau$ and $\boldsymbol{\tilde{\tau}}$ is a letter-onto, everywhere growing, proper and recognizable directive sequence generating $Y$. The sequence $\boldsymbol{\tilde{\tau}}$ has all the properties required by the theorem but having alphabet rank bounded by $K$. To overcome this, we use Proposition \ref{prop:decrese_rank_factor} with $\boldsymbol{\tilde{\phi}}$ and do more contractions to obtain a letter-onto and proper factor $\boldsymbol{\phi}\colon \boldsymbol{\tilde{\sigma}} \to\boldsymbol{\tau}$ such that $\phi_0 = \tilde{\phi}_0 = \tau$ and $\boldsymbol{\tau}$ is a letter-onto, everywhere growing, proper and recognizable directive sequence generating $Y$ and satisfying $\mathrm{AR}(\boldsymbol{\tau}) \leq \mathrm{AR}(\tilde{\boldsymbol{\sigma}}) = \mathrm{AR}(\boldsymbol{\sigma})$.

It left to prove the last part of the theorem. Observe that since $\boldsymbol{\tilde{\sigma}}$ and $\boldsymbol{\sigma}$ differ only at their first coordinate, $\boldsymbol{\phi}$ is also a factor from $\boldsymbol{\sigma}$ to $\boldsymbol{\tau}$. Further, by equation \eqref{eq:theo:rank_of_factor:1} and the fact that $\phi_0 = \tau$, we have
$\pi(\sigma_0(x)) = \tau(x) = \phi_0(x)$ and $|\sigma_0(a)| =  |\phi_0(a)|$ for every $x \in X_{\boldsymbol{\sigma}}^{(1)}$ and $a \in \cA_1$.
\end{proof}

\begin{cor}\label{cor:non_reco_to_reco_minimal_case}
Let $(X,T)$ be an aperiodic minimal subshift of generated by an everywhere growing and proper directive sequence of alphabet rank $K$. Then, the topological rank of $X$ is at most $K$.
\end{cor}
\begin{proof}
We can use Theorem \ref{theo:rank_of_factor} to obtain an everywhere growing, proper and recognizable directive sequence $\boldsymbol{\tau} = (\tau_n\colon\cB_{n+1}^+\to\cB_n^+)_{n\in\N}$ generating $X$ and having of alphabet rank at most $K$. Due to Lemma \ref{lem:non_letter-onto_to_letter-onto}, we can assume that $\boldsymbol{\tau}$ is letter-onto. In particular, $\cB_n \subseteq \cL(X_{\boldsymbol{\tau}}^{(n)})$ for every $n \in \N$.

We claim that $X_{\boldsymbol{\tau}}^{(n)}$ is minimal. Indeed, if $Y \subseteq X_{\boldsymbol{\tau}}^{(n)}$ is a subshift, then $\tau_{[0,n)}(Y)$ is closed (as $\tau_{[0,n)}\colon X_{\boldsymbol{\tau}}^{(n)}\to X_{\boldsymbol{\tau}}$ is continuous), so $\bigcup_{k\in\Z}T^k\tau_{[0,n)}(Y) = \bigcup_{|k|\leq|\tau_{[0,n)}|}T^k\tau_{[0,n)}(Y)$ is a subshift in  $X_{\boldsymbol{\tau}}$ which, by minimality, is equal to it. Thus, any point $x \in X_{\boldsymbol{\tau}}^{(n)}$ has a $\tau_{[0,n)}$-factorization $(k,z)$ with $z \in Y$. The recognizability property of $(X_{\boldsymbol{\tau}}^{(n)},\tau_{[0,n)})$ then implies that $Y = X_{\boldsymbol{\tau}}^{(n)}$.

Now, we prove that for any $n\in \N$ there exists $N>n$ such that $\tau_{[n,N)}$ is positive. This would imply that the topological rank of $X$ is at most $K$ and hence would complete the proof. Let $n \in \N$ and $R$ be a constant of recognizability for $(X_{\boldsymbol{\tau}}^{(n)},\tau_{[0,n)})$. Since $X_{\boldsymbol{\tau}}^{(n)}$ is minimal, there exists a constant $L\geq1$ such that two consecutive occurrences of a word $w \in \cL(X_{\boldsymbol{\tau}}^{(n)})\cap\cB_n^{2R+1}$ in a point $x \in X_{\boldsymbol{\tau}}^{(n)}$ are separated by at most $L$. Let $N > n$ be big enough so that $\langle\tau_{[0,N)}\rangle \geq L+2R$. Then, for all $a \in \cB_N \subseteq \cL(X_{\boldsymbol{\tau}}^{(N)})$ and $w \in \cL(X_{\boldsymbol{\tau}}^{(n)})\cap\cB_n^{2R+1}$, $w$ occurs at a position $i \in \{R,R+1,\dots,|\tau_{[0,N)}(a)|-R\}$ of $\tau_{[0,N)}(a)$. Since $R$ is a recognizability constant for $(X_{\boldsymbol{\tau}}^{(n)},\tau_{[0,n)})$, we deduce that for all $a \in \cB_N$ and $b \in \cB_n$, $b$ occurs in $\tau_{[n,N)}(a)$. Thus, $\tau_{[n,N)}$ is positive.
\end{proof}

We can now prove Corollary \ref{cor:factor_nonproper}.
\begin{ccor}{\ref{cor:factor_nonproper}}
Let $(X,T)$ be an aperiodic minimal subshift generated by an everywhere growing directive sequence of finite alphabet rank. Then, the topological rank of $(X,T)$ is finite.
\end{ccor}
\begin{proof}
We are going to prove that $X$ is generated by an everywhere growing and proper directive sequence $\boldsymbol{\tau}$ of finite alphabet rank. This would imply, by Corollary \ref{cor:non_reco_to_reco_minimal_case}, that the topological rank of $X$ is finite. Let $\boldsymbol{\sigma} = (\sigma_n\colon\cA_{n+1}^+\to\cA_n^+)_{n\in\N}$ be an everywhere growing directive sequence of finite alphabet rank generating $X$. We contract $\boldsymbol{\tau}$ in a way such that $\#\cA_n \leq K$ for every $n\geq1$.

We are going to inductively define subshifts $X_n$, $n\in\N$. We start with $X_0 \coloneqq X$. We now assume that $X_n$ is defined for some $n \in \N$. Then the set $X'_{n+1} = \{x \in X_{\boldsymbol{\sigma}}^{(n+1}: \sigma_n(x) \in X_n\}$ is a subshift. We define $X_{n+1}$ as any minimal subshift contained in $X'_{n+1}$. It follows from the definition of $X_{n+1}$ that $\bigcup_{k\in\Z}T^k\sigma_n(X_{n+1}) \subseteq X_n$. Being $X_n$ minimal, we have
\begin{equation}\label{eq:cor:factor_nonproper:1}
\text{$\bigcup_{k\in\Z}T^k\sigma_n(X_{n+1}) = X_n$.}
\end{equation}
Let $\tilde{\cA}_n = \cA_n \cap \cL(X_n)$. Equation \eqref{eq:cor:factor_nonproper:1} and the fact that $\boldsymbol{\sigma}$ is everywhere growing allow us to assume without loss of generality that, after a contraction of $\boldsymbol{\sigma}$, the following holds for every $n\in\N$:
\begin{equation}\label{eq:cor:factor_nonproper:1.5}
\text{if $a\in\tilde{\cA}_{n+1}$ and $w \in \cL(X_n)$ has length $3$, then $w$ occurs twice in $\sigma_n(a)$.}
\end{equation}
Let us fix a word $w_n = a_nb_nc_b \in \cL(X_n)$ of length $3$. Then, by \eqref{eq:cor:factor_nonproper:1.5}, we can decompose $\sigma_n(a) = u_n(a)v_n(a)$ in a way such that
\begin{equation}\label{eq:cor:factor_nonproper:2}
\text{$u_n(a)$ ends with $a_n$, $v_n(a)$ starts with $b_nc_n$ and $|v_n(a)| \geq 2$.}
\end{equation}

To define $\boldsymbol{\tau}$, we need to introduce additional notation first. Let $\cB_n$ be the alphabet consisting of tuples ${a\brack b}$ such that $ab \in \cL(X_n)$. Also, if $w = w_1\dots w_{|w|} \in \cL(X_n)$ has length $|w|\geq2$, then $\chi_n(w) \coloneqq {w_1\brack w_2} {w_2\brack w_3}\dots {w_{|w|-1}\brack w_{|w|}} \in \cB_n^+$, and if $w' = {w_1\brack w_2}\dots {w_{|w|-1}\brack w_{|w|}} \in \cB_0^+$, then $\eta(w') \coloneqq w_1\dots w_{|w|-1} \in \cA_0^+$. Observe that $\eta\colon\cB_0^+\to\cA_0^+$ is a morphism.

We now define $\boldsymbol{\tau}$. Let $\tau_n\colon\cB_{n+1}^+\to\cB_n^+$ be the unique morphism such that $\tau_n({a\brack b}) = \chi_n(v_n(a)u_n(a)b_n)$ for every ${a\brack b} \in \cB_{n+1}$. Observe that since $v_n(a)u_n(a)b_n \in \cL(X_n)$, it is indeed the case that $\tau_n({a\brack b})\in\cB_n^+$. We set $\boldsymbol{\tau} = (\eta\tau_0,\tau_1,\tau_2,\dots)$.

It follows from \eqref{eq:cor:factor_nonproper:2} that for every $n \in \N$ and ${a\brack b} \in \cB_{n+1}$, $\tau_n({a\brack b})$ starts with ${b_n\brack c_n}$ and ends with ${a_n \brack b_n}$. Thus, $\boldsymbol{\tau}$ is proper. Moreover, since $|v_n(a)| \geq 2$, we have $|v_n(a)u_n(a)b_n| \geq 3$ and thus $|\tau_n({a\brack b})| \geq 2$. Therefore, $\langle\tau_n\rangle \geq 2$ and $\boldsymbol{\tau}$ is everywhere growing. Also, $\#\cB_n \leq \#\cA_n^2 \leq K^2$ for every $n \in \N$, so the alphabet rank of $\boldsymbol{\tau}$ is finite.

It remains to prove that $X = X_{\boldsymbol{\tau}}$. By minimality, it is enough to prove that $X \supseteq X_{\boldsymbol{\tau}}$. Observe that since $\tau_n\chi_{n+1}(ab) = \chi_n(v_n(a)u_n(b)b_n)$, the word $\tau_n\chi_{n+1}(ab)$ occurs in $\chi_n\sigma_n(ab)$. Moreover, for every $w = w_1\dots w_{|w|} \in \cL(X_{\boldsymbol{\sigma}}^{(n)})$, $\tau_n\chi_{n+1}(w)$ occurs in $\chi_n\sigma_n(w)$. Then, by using the symbol $\sqsubseteq$ to denote the ``subword'' relation, we can write for every $n \in \N$ and $ab \in \cL(X_{\boldsymbol{\sigma}}^{(n)})$:
\begin{align*}
\tau_{[0,n)}\chi_n(ab) &\sqsubseteq
\tau_{[0,n-1)}\chi_{n-1}\sigma_{n-1}(ab) \\ &\sqsubseteq 
\tau_{[0,n-2)}\chi_{n-2}\sigma_{[n-2,n)}(ab) \sqsubseteq 
\dots \sqsubseteq
\chi_0\sigma_{[0,n)}(ab)
\end{align*}
Hence, $\eta\tau_{[0,n)}({a\brack b}) \sqsubseteq \eta\chi_0\sigma_{[0,n)}(ab) \sqsubseteq \sigma_{[0,n)}(ab)$. We conclude that $X_{\boldsymbol{\tau}} \subseteq X_{\boldsymbol{\sigma}} = X$.
\end{proof}

\begin{cor}\label{cor:question_1}
Let $(X,T)$ be a minimal subshift of topological rank $K$ and $\pi\colon(X,T)\to(Y,T)$ a factor map, where $Y$ is an aperiodic subshift. Then, the topological rank of $Y$ is at most $K$.
\end{cor}
\begin{proof}
By Theorem \ref{theo:FTR_from_DDPM}, $(X,T)$ is generated by a proper and primitive directive sequence $\boldsymbol{\sigma}$ of alphabet rank equal to $K$. In particular, $\boldsymbol{\sigma}$ is everywhere growing and proper, so we can use Theorem \ref{theo:rank_of_factor} to obtain an everywhere growing, proper and recognizable directive sequence $\boldsymbol{\tau} = (\tau_n\colon\cB_{n+1}^+\to\cB_n^+)_{n\geq0}$ generating $(Y,T)$ and having of alphabet rank at most $K$. Then, the hypothesis of Corollary \ref{cor:non_reco_to_reco_minimal_case} hold for $(Y,T)$, and thus the topological rank of $(Y,T)$ is at most $K$.
\end{proof}

The following notion will be used in the proof of the theorem below: $\boldsymbol{\sigma} = (\sigma_n\colon\cA_{n+1}^+\to\cA_n)_{n\geq0}$ has {\em exact alphabet rank at most $K$} if $\#\cA_n\leq K$ for all $n\geq 1$.
\begin{ccor}{\ref{theo:intro:2}}\label{theo:coalescency_chain_factors}
Let $(X,T)$ be an $\cS$-adic subshift generated by an everywhere growing and proper sequence of alphabet rank $K$, and $\pi_j\colon(X_{j+1},T) \to (X_j,T)$, $j=0,\dots,L$, be a chain of aperiodic symbolic factors, with $X_L = X$. Suppose that $L > \log_2(K)$. Then $\pi_j$ is a conjugacy for some $j$.
\end{ccor}
\begin{proof}
We start by using Theorem \ref{theo:rank_of_factor} with the identity function $\mathrm{id}\colon (X,T)\to (X,T)$ to obtain a letter-onto, everywhere growing, proper and recognizable directive sequence $\boldsymbol{\sigma_L}$ of alphabet rank at most $K$ generating $X$. By doing a contraction, we can assume that $\boldsymbol{\sigma_L}$ has exact alphabet rank at most $K$.

By Theorem \ref{theo:rank_of_factor} applied to $\pi_{L-1}$ and $\boldsymbol{\sigma_L}$, there exists, after a contraction of $\boldsymbol{\sigma_L}$, a letter-onto factor
$\boldsymbol{\phi_{L-1}}\colon\boldsymbol{\sigma_L}\to\boldsymbol{\sigma_{L-1}}$, where $\boldsymbol{\sigma_{L-1}}$ is letter-onto, everywhere growing, proper, recognizable, has alphabet rank at most $K$, generates $X_{L-1}$, and, if $\phi_{L-1,0}$ and $\sigma_{L,0}$ are the first coordinates of $\boldsymbol{\phi_{L-1}}$ and $\boldsymbol{\sigma_{L}}$, respectively, then $\pi_{L-1}(\sigma_{L,0}(x)) = \phi_{L-1,0}(x)$ for every $x \in X_{\boldsymbol{\sigma_L}}^{(1)}$ and $|\sigma_{L,0}(a)| = |\phi_{L-1,0}(a)|$ for every letter $a$ in the domain of $\sigma_{L,0}$. By contracting these sequences, we can also suppose that $\boldsymbol{\sigma_{L-1}}$ has exact alphabet rank at most $K$. The same procedure applies to $\pi_{L-2}$ and $\boldsymbol{\sigma_{L-1}}$. Thus, by continuing in this way we obtain for every $j=0,\dots,L-1$ a letter-onto factor $\boldsymbol{\phi_j}\colon\boldsymbol{\sigma_{j+1}}\to\boldsymbol{\sigma_j}$ such that
\begin{enumerate}
\item[$\bullet$] $\boldsymbol{\sigma_j}$ is letter-onto, everywhere growing, proper, recognizable, has exact alphabet rank at most $K$, generates $X_j$, $\pi_{j}(\sigma_{j+1,0}(x)) = \phi_{j,0}(x)$ for every $x \in X_{\boldsymbol{\sigma_{j+1}}}^{(1)}$, and $|\sigma_{j+1,0}(a)| = |\phi_{j,0}(a)|$ for every $a \in \cA_{j+1,1}$.
\end{enumerate}
Here, we are using the notation $\boldsymbol{\sigma_j} = (\sigma_{j,n}\colon\cA_{j,n+1}^+\to\cA_{j,n}^+)_{n\in\N}$, $\boldsymbol{\phi_j} = (\phi_{j,n}\colon\cA_{j+1,n}^+\to\cA_{j,n}^+)_{n\in\N}$ and $X_j^{(n)} = X_{\boldsymbol{\sigma_j}}^{(n)}$. We note that
\begin{enumerate}
\item[$(\triangle_1)$] for every $x \in X_{j+1}^{(1)}$, $\pi_j(\sigma_{j+1,0}(x)) = \phi_{j,0}(x) = \sigma_{j,0}\phi_{j,1}(x)$ since $\phi_{j,0} = \sigma_{j,0}\phi_{j,1}$;
\item[$(\triangle_2)$] $X_j^{(1)} = \bigcup_{k\in\Z}T^k\phi_{j,1}(X_{j+1}^{(1)})$ by Lemma \ref{lem:regularity_factor_Ssequences}.
\end{enumerate}
Hence, the following diagram commutes:
\begin{equation*}
\begin{tikzcd}
X_0^{(1)} \arrow{d}{\sigma_{0,0}} &
\quad \cdots \quad X_j^{(1)} \arrow[xshift=4ex]{d}{\sigma_{j,0}} \arrow{l}[swap]{\phi_{0,1}}&
X_{j+1}^{(1)} \arrow[xshift=-4ex]{d}{\sigma_{j+1,0}} \arrow{l}[swap]{\phi_{j,1}}
\quad \cdots\quad  &
X_L^{(1)} \arrow{d}{\sigma_{L,0}} \arrow{l}[swap]{\phi_{L-1,1}}\\
X_0^{(0)}&
\quad \cdots \quad X_j^{(0)} \arrow{l}{\pi_{0}} &
X_{j+1}^{(0)} \arrow{l}{\pi_j} 
\quad \cdots\quad  &
X_L^{(0)} \arrow{l}{\pi_{L-1}}
\end{tikzcd}
\end{equation*}

\begin{claim}
If $(X_{j+1}^{(1)}, \phi_{j,1})$ is recognizable, then $\pi_j$ is a conjugacy.
\end{claim}
\begin{claimproof}
Let us assume that $(X_{j+1}^{(1)}, \phi_{j,1})$ is recognizable and let, for $i=0,1$, $x^i \in X_{j+1}^{(1)}$ such that $y = \pi_j(x^0) = \pi_j(x^1)$. We have to show that $x^0 = x^1$. First, we use Lemma \ref{lem:desubstitution} to find a centered $\sigma_{j+1,0}$-factorization $(k^i,z^i)$ of $x^i$ in $X_{j+1}^{(1)}$. Then, equation $\triangle_1$ allows us to compute:
\begin{equation*}
T^{k^0}\sigma_{j,0}\phi_{j,1}(z^0) = T^{k^0}\pi_j(\sigma_{j+1,0}(z^0)) =
\pi_j(x^0) = \pi_j(x^1) = T^{k^1}\sigma_{j,0}\phi_{j,1}(z^1).
\end{equation*}
This implies that $(k^i,z^i)$ is a $\sigma_{j,0}\phi_{j,1}$-factorization of $y$ in $X_{j+1}^{(1)}$ for $i=0,1$. Moreover, these are centered factorizations as, by $\bullet$, $|\sigma_{j,0}\phi_{j,1}(a)| = |\sigma_{j+1,0}(a)|$ for all $a \in \cA_{j+1,1}$. Now, being $(X_j^{(1)},\sigma_{0,j})$ and $(X_{j+1}^{(1)}, \phi_{j,1})$ recognizable, Lemma \ref{lem:recognizability_of_composition} gives that $(X_{j+1}^{(1)},\sigma_{j,1}\phi_{j,1})$ is recognizable, and thus we have that $(k^0,z^0) = (k^1,z^1)$. Therefore, $x^0 = x^1$ and $\pi$ is a conjugacy. 
\end{claimproof}

Now we can finish the proof. We assume, by contradiction, that $\pi_j$ is not a conjugacy for all $j$. Then, by the claim,
\begin{equation}\label{eq:theo:coalescency_chain_factors:1}
\text{$(X_j^{(1)},\phi_{1,j})$ is not recognizable for every $j \in \{0,\dots,L-1\}$.}
\end{equation}
Let
\begin{equation*}
\boldsymbol{\nu} = (\phi_{0,1}, \phi_{1,1}, \phi_{2,1}, \dots,\phi_{L-1,1},\sigma_{L,1}, \sigma_{L,2}, \sigma_{L,3}, \dots).
\end{equation*}
The idea is to use Theorem \ref{theo:finite_rank_is_expansive} with $\boldsymbol{\nu}$ to obtain a contradiction. To do so, we first note that, since $\boldsymbol{\nu}$ and $\boldsymbol{\sigma}^{(L)}$ have the same ``tail'', $X_{\boldsymbol{\nu}}^{(m+L)} = X_{L}^{(m+1)}$ for all $m \in \N$. Moreover, $\triangle_2$ and the previous relation imply that
\begin{align*}
X_{\boldsymbol{\nu}}^{(j)} &=
\bigcup_{k\in\Z} T^k\phi_{j,1}(X_{\boldsymbol{\nu}}^{(j+1)}) = \dots =
\bigcup_{k\in\Z} T^k\phi_{j,1}\dots\phi_{L-1,1}(X_{\boldsymbol{\nu}}^{(L)}) \\ &=
\bigcup_{k\in\Z} T^k\phi_{j,1}\dots\phi_{L-1,1}(X_{L}^{(1)}) =
\bigcup_{k\in\Z} T^k\phi_{j,1}\dots\phi_{L-2,1}(X_{L-1}^{(1)}) = \dots = X_j^{(1)}.
\end{align*}
This and \eqref{eq:theo:coalescency_chain_factors:1} imply that for every $j \in \{1,\dots,L-1\}$, the level $(X_{\boldsymbol{\nu}}^{(j)}, \phi_{j,1})$ of $\boldsymbol{\nu}$ is not recognizable. Being $\boldsymbol{\nu}$ everywhere growing as $\boldsymbol{\sigma_L}$ has this property, we conclude that Theorem \ref{theo:finite_rank_is_expansive} can be applied and, therefore, that $X_0^{(1)} = X_{\boldsymbol{\nu}}$ is periodic. But then $X_0 = \bigcup_{k\in\Z}T^k\sigma_{0,0}(X_0^{(1)})$ is periodic, contrary to our assumptions.
\end{proof}

Recall that a system $(X,T)$ is coalescent if every endomorphism $\pi\colon(X,T)\to(X,T)$ is an automorphism. 
\begin{cor}\label{theo:coalescency}
Let $(X,T)$ be an $\cS$-adic subshift generated by an everywhere growing and proper directive sequence of finite alphabet rank. Then, $(X,T)$ is coalescent.
\end{cor}

\begin{rem}\label{rem:coalescency}
A linearly recurrent subshift of constant $C$ is generated by a primitive and proper directive sequence of alphabet rank at most $C(C+1)^2$ (\cite{durand_2000}, Proposition 6). In \cite{DHS99}, the authors proved the following
\begin{theo}[\cite{DHS99}, Theorem 3]
\emph{ For a linearly recurrent subshift $X$ of constant $C$, in any chain of factors $\pi_j\colon(X_j,T)\to(X_{j+1},T)$, $j=0,\dots,L$, with $X_0 = X$ and $L \geq (2C(2C+1)^2)^{4C^3(2C+1)^2}$ there is at least one $\pi_j$ which is a conjugacy. }
\end{theo}
Thus, Theorem \ref{theo:coalescency_chain_factors} is not only a generalization of this result to a much larger class of systems, but also improves the previous super-exponential constant to a logarithmic one.
\end{rem}

In Proposition 28 of \cite{DHS99}, the authors proved that Cantor factors of linearly recurrent systems are either subshifts or odometers. Their proof only uses that this kind of systems satisfy the strong coalescence property that we proved in Corollary \ref{theo:coalescency} for finite topological rank systems. Therefore, by the same proof, we have:
\begin{cor}\label{cor:cantor_factors}
Let $\pi\colon(X,T) \to(Y,T)$ be a factor map between minimal systems. Assume that $(X,T)$ has finite topological rank and that $(Y,T)$ is a Cantor system. Then, $(Y,T)$ is either a subshift or a odometer.
\end{cor}
\begin{proof}
We sketch the proof from \cite{DHS99} that we mentioned above.

Let $(\cP_n)_{n\in\N}$ be a sequence of clopen partitions of $Y$ such that $\cP_{n+1}$ is finer than $\cP_n$ and their union generates the topology of $Y$. Also, let $Y_n$ be the subshift obtained by codifying the orbits of $(Y,T)$ by using the atoms of $\cP_n$. Then, the fact that $\cP_n$ is a clopen partition induces a factor map $\pi_n\colon(Y,T) \to (Y_n,T)$. Moreover, since $\cP_{n+1}$ is finer than $\cP_n$, there exists a factor map $\xi_n\colon(Y_{n+1},T)\to(Y_n,T)$ such that $\xi_n\pi_{n+1} = \pi_n$. Hence, we have the following chain of factors:
\begin{equation*}
(X,T) \overset{\pi}{\longrightarrow}
(Y,T) \overset{\pi_n}{\longrightarrow}
(Y_n,T) \overset{\xi_{n-1}}{\longrightarrow}
(Y_{n-1},T) \overset{\xi_{n-2}}{\longrightarrow} \dots
\overset{\xi_{1}}{\longrightarrow} (Y_0,T).
\end{equation*}
We conclude, by also using the fact that the partitions $\cP_n$ generate the topology of $Y$, that $(Y,T)$ is conjugate to the inverse limit $\overleftarrow{\lim}_{n\to\infty}(Y_n;\xi_n)$.

Now we consider two cases. If $Y_n$ is periodic for every $n \in \N$, then $Y$ is the inverse limit of periodic system, and hence an odometer. In the other case, we have, by Corollary \ref{theo:coalescency_chain_factors}, that $\xi_n$ is a conjugacy for all big enough $n \in \N$, and thus that  $(Y,T)$ is conjugate to one of the subshifts $Y_n$.
\end{proof}

\section{Fibers of symbolic factors}\label{sec:fibers}

The objective of this section is to prove Theorem \ref{theo:factores_k_1}, which states that factor maps $\pi\colon(X,T)\to(Y,T)$ between $\cS$-adic subshifts of finite topological rank are always {\em almost $k$-to-$1$} for some $k$ bounded by the topological rank of $X$. We start with some lemmas from topological dynamics.

\begin{lem}[\cite{auslander}]\label{lemma:continuidad_g_delta}
Let $\pi\colon X \to Y  $ be a continuous map between compact metric spaces. Then $\pi^{-1}\colon Y \to 2^X$ is continuous at every point of a residual subset of $Y$.
\end{lem}

Next lemma gives a sufficient condition for a factor map $\pi$ to be almost $k$-to-1. Recall that $E(X,T)$ stands for the Ellis semigroup of $(X,T)$.
\begin{lem}\label{lem:pseudo_highly_proximal}
Let $\pi\colon (X,T) \to (Y,T)$ be a factor map between topological dynamical systems, with $(Y,T)$ minimal, and $K\geq1$ an integer. Suppose that for every $y \in Y$ there exists $u \in E(2^X,T)$ such that $\#u\circ \pi^{-1}(y) \leq K$. Then, $\pi$ is almost $k$-to-1 for some $k\leq K$.
\end{lem}
\begin{proof}
First, we observe that by the description of $u\circ A$ in terms of nets at the end of Subsection \ref{subsec:topdyn}, we have
\begin{equation}\label{eq:lem:pseudo_highly_proximal:1}
	\# u\circ A \leq \#A,\ \forall u\in E(2^X,T),\ A\in 2^X.
\end{equation}
Now, by previous lemma, there exists a residual set $\tilde{Y} \subseteq Y$ of continuity points for $\pi^{-1}$. Let $y,y'\in\tilde{Y}$ be arbitrary. Since $Y$ is minimal, there exists a sequence $(n_\ell)_\ell$ such that $\lim_\ell T^{n_\ell}y = y'$. If $w \in E(2^X,T)$ is the limit of a convergent subnet of $(T^{n_\ell})_\ell$, then 
$wy = y'$. By the continuity of $\pi^{-1}$ at $y'$ and \eqref{eq:lem:pseudo_highly_proximal:1}, we have
\begin{equation*}
	\#\pi^{-1}(y') = \#\pi^{-1}(wy) =\#w\circ \pi^{-1}(y) \leq \#\pi^{-1}(y).
\end{equation*}
We deduce, by symmetry, that $\#\pi^{-1}(y') = \#\pi^{-1}(y)$. Hence, $k\coloneqq \pi^{-1}(y)$ does not depend on the chosen $y\in \tilde{Y}$. To end the proof, we have to show that $k\leq K$. We fix $y \in \tilde{Y}$ and take, using the hypothesis, $u \in E(2^X,T)$ such that $\#u\circ\pi^{-1}(y) \leq K$. As above, by minimality, there exists $v\in E(2^X,T)$ such that $vuy = y$. Then, by the continuity of $\pi^{-1}$ at $y$,
\begin{equation*}
\pi^{-1}(y) = \pi^{-1}(vuy) = (vu)\circ\pi^{-1}(y) = v\circ (u\circ\pi^{-1}(y)).
\end{equation*}
This and \eqref{eq:lem:pseudo_highly_proximal:1} imply that $k = \#\pi^{-1}(y) \leq \#u\circ\pi^{-1}(y) \leq K$.
\end{proof}

Let $\sigma\colon\cA^+\to\cB^+$ be a morphism, $(k,x)$ a centered $\sigma$-factorization of $y \in \cB^\Z$ in $\cA^\Z$ and $\ell \in \Z$. Note that there exists a unique $j \in \Z$ such that $\ell \in [c_{\sigma,j}(k,x), c_{\sigma,j+1}(k,x))$ (recall the notion of \emph{cut} from Definition \ref{defi:factorizations}). In this context, we say that $(c_{\sigma,j}(k,x),x_j)$ is \emph{the symbol of $(k,x)$ covering position $\ell$ of $y$}.
\begin{ctheo}{\ref{theo:intro:3}}\label{theo:factores_k_1}
Let $\pi\colon(X,T) \to (Y,T)$ be a factor between subshifts, with $(Y,T)$ minimal and aperiodic. Suppose that $X$ is generated by a proper and everywhere growing directive sequence $\boldsymbol{\sigma}$ of alphabet rank $K$. Then, $\pi$ is almost $k$-to-1 for some $k\leq K$.
\end{ctheo}
\begin{proof}
Let $\boldsymbol{\sigma} = (\sigma_n\colon\cA_{n+1}\to\cA_n)_{n\geq0}$ be a proper and everywhere growing directive sequence of alphabet rank at most $K$ generating $X$. Due the possibility of contracting $\boldsymbol{\sigma}$, we can assume without loss of generality that $\#\cA_{n} \leq K$ for every $n\geq1$ and that $\sigma_0$ is $r$-proper, where $r$ is the radius of $\pi$. Then, by Lemma \ref{lem:factor_to_morphism}, $Y$ is generated by an everywhere growing directive sequence of the form $\boldsymbol{\tau} = (\tau,\sigma_1,\sigma_2,\dots)$, where $\tau\colon\cA_1^+\to\cB^+$ is such that $\tau(x) = \pi(\sigma_0(x))$ for every $x\in X_{\boldsymbol{\tau}}^{(1)} = X_{\boldsymbol{\sigma}}^{(1)}$.
We will use the notation $\tau_{[0,n)} = \tau\sigma_{[1,n)}$. Further, for $y \in Y$ and $n\geq1$, we write $F_n(y)$ to denote the set of $\tau_{[0,n)}$-factorizations of $y$ in $Y_{\boldsymbol{\tau}}^{(n)}$. 
\begin{claim}
There exist $\ell_n\in\Z$ and $G_n \subseteq \Z\times\cB_{n+1}$ with at most $K$ elements such that if $(k,x)\in F_n(y)$, then the symbol of $(k,x)$ covering position $\ell_n$ of $y$ is in $G_n$.
\end{claim}
\begin{claimproof}
First, since $Y$ is aperiodic, there exists $L \in \N$ such that 
\begin{equation}\label{eq:theo:factores_k_1:1}
\text{all words $w \in \cL(Y)$ of length $\geq L$ have least period greater than $|\tau_{[0,n)}|$}.
\end{equation}
We assume, by contradiction, that the claim does not hold. In particular, for every $\ell\in [0,L)$ there exist $K+1$ $\tau_{[0,n)}$-factorizations $(x,k)$ of $y$ in $Y_{\boldsymbol{\tau}}^{(n)}$ such that their symbols covering position $\ell$ of $y$ are all different. Now, since $\#\tau_{[0,n)}(\cA_{n+1}) \leq K$, we can use the Pigeon Principle to find two of such factorizations, say $(k,x)$ and $(k',x')$, such that if $(c,a)$ and $(c',a')$ are their symbols covering position $\ell$ of $y$ then $a = a'$ and $c<c'$. Then,
\begin{equation*}
y_{(c,c+|\tau_{[0,n)}(a)|]} = \tau_{[0,n)}(a) = y_{(c',c'+|\tau_{[0,n)}(a)|]}
\end{equation*}
and, thus, $y_{(c,c'+|\tau_{[0,n)}(a)|]}$ is $(c'-c)$-periodic. Being $\ell \in (c',c+|\tau_{[0,n)}(a)|)$, we deduce that the local period of $y_{[0,L)}$ at $\ell$ is at most $c'-c \leq |\tau_{[0,n)}|$. Since this true for every $\ell \in [0,L)$ and since, by Theorem \ref{theo:factorizacion_critica}, $\per(y_{[0,L)}) = \per(y_{[0,L)}, y_{[0,\ell)})$ for some $\ell \in [0,L)$, we conclude that $\per(y_{[0,L)}) \leq |\tau_{[0,n)}|$. This contradicts \eqref{eq:theo:factores_k_1:1} and proves thereby the claim.
\end{claimproof}

Now we prove the theorem. It is enough to show that the hypothesis of Lemma \ref{lem:pseudo_highly_proximal} hold. Let $y \in Y$ and $\tilde{F}_n(y) \subseteq F_n(y)$ be such that $\#\tilde{F}_n(y) = \#G_n$ and the set consisting of all the symbols of factorizations $(k,x) \in \tilde{F}_n(y)$ covering position $\ell_n$ of $y$ is equal to $G_n$. Let $z \in \pi^{-1}(y)$ and $(k,x)$ be a $\sigma_{[0,n)}$-factorization of $z$ in $X_{\boldsymbol{\sigma}}^{(n)}$. Then, $T^k\tau_{[0,n)}(x) = T^k\pi(\sigma_{[0,n)}(x)) = \pi(z) = y$ and $(k,x)$ is a $\tau_{[0,n)}$-factorization of $y$ in $Y_{\boldsymbol{\tau}}^{(n)}$. Thus, we can find $(k',x') \in \tilde{F}_n(y)$ such that the symbols of $(k,x)$ and $(k',x')$ covering  position $\ell_n$ of $y$ are the same; let $(m,a)$ be this common symbol. Since $\boldsymbol{\sigma}$ is proper, we have
\begin{equation*}
z_{[m-\langle\sigma_{[0,n-1)}\rangle,m+|\sigma_{[0,n)}(a)|+\langle\sigma_{[0,n-1)}\rangle]} =
z'_{[m-\langle\sigma_{[0,n-1)}\rangle,m+|\sigma_{[0,n)}(a)|+\langle\sigma_{[0,n-1)}\rangle]},
\end{equation*}
where $z' = T^{k'}\sigma_{[0,n)}(x') \in X$ is the point that $(k',x')$ factorizes in $(X_{\boldsymbol{\sigma}}^{(n)}, \sigma_{[0,n)})$. Then, as $\ell_n \in (m, m+|\sigma_{[0,n)}(a)|]$,
\begin{equation*}
	z_{(\ell_n-\langle\sigma_{[0,n-1)}\rangle, \ell_n+\langle\sigma_{[0,n-1)}\rangle]} =
	z'_{(\ell_n-\langle\sigma_{[0,n-1)}\rangle, \ell_n+\langle\sigma_{[0,n-1)}\rangle]}.
\end{equation*}
Thus, $\mathrm{dist}(T^{\ell_n}z, T^{\ell_n}P_n(y)) \leq \exp(-\langle\sigma_{[0,n-1)}\rangle)$, where $P_n(y)\subseteq\pi^{-1}(y)$ is the set of all points $T^{k''}\sigma_{[0,n)}(x'') \in X$ such that $(k'',x'') \in \tilde{F}_n(y)$. Since this holds for every $n \geq 1$, we obtain that $d_\mathrm{H}(T^{\ell_n}\pi^{-1}(y),T^{\ell_n} P_n(y))$ converges to zero as $n$ goes to infinity (where, we recall, $d_\mathrm{H}$ is the Hausdorff distance). By taking an appropriate convergent subnet $u \in E(2^X,T)$ of $(T^{\ell_n})_{n\in\N}$ we obtain $\# u\circ \pi^{-1}(y) \leq \sup_{n\in\N}\#P_n = \sup_{n\in\N}\#G_n \leq K$. This proves that the hypothesis of Lemma \ref{lem:pseudo_highly_proximal} holds. Therefore, $\pi$ is almost $k$-to-1 for some $k\leq K$.
\end{proof}

\section{Number of symbolic factors}\label{sec:number_factors}

In this section we prove Theorem \ref{theo:intro:4}. In order to do this, we split the proof into 3 subsections. First, in Lemma \ref{lem:many_distals_implies_conjugacy} of subsection \ref{subsec:number_factors:distal_case}, we deal with the case of Theorem \ref{theo:intro:4} in which the factor maps are distal.
Next, we show in Lemma \ref{lem:not_distal_implies_recursion} from Subsection \ref{subsec:number_factors:non-distal_case} that in certain technical situation -which will arise when we consider non-distal factor maps- it is possible to reduce the problem to a similar one, but where the alphabet are smaller. Then, we prove Theorem \ref{theo:intro:4} in subsection \ref{subsec:number_factors:main_result} by a repeated application of the previous lemmas.

\subsection{Distal factor maps}\label{subsec:number_factors:distal_case}
We start with some definitions. If $(X,T)$ is a system, then we always give $X^k$ the diagonal action $T^{[k]} \coloneqq T\times\cdots\times T$. If $\pi\colon(X,T)\to(Y,T)$ is a factor map and $k \geq 1$, then we define $R_\pi^k = \{(x^1,\dots,x^k)\in X^k : \pi(x^1) = \cdots =\pi(x^k)\}$. Observe that $R_\pi^k$ is a closed $T^{[k]}$-invariant subset of $X^k$. 

Next lemma follows from classical ideas from topological dynamics. See, for example, Theorem 6 in Chapter 10 of \cite{auslander}.

\begin{lem}\label{lem:kto1_to_groupextension}
Let $\pi\colon(X,T)\to(Y,T)$ be a distal almost $k$-to-1 factor between minimal systems, $z = (z^1,\dots,z^k) \in R_\pi^k$ and $Z = \overline{\mathrm{orb}}_{T^{[k]}}(z)$. Then, $\pi$ is $k$-to-1 and $Z$ is minimal
\end{lem}

We will also need the following lemma:
\begin{lem}[\cite{durand_2000}, Lemma 21]\label{lem:fabien}
Let $\pi_i\colon (X,T) \to (Y_i,T)$, $i=0,1$, be two factors between aperiodic minimal systems. Suppose that $\pi_0$ is finite-to-1. If $x,y\in X$ are such that $\pi_0(x) = \pi_0(y)$ and $\pi_1(x) = T^p\pi_1(y)$, then $p=0$.
\end{lem}

\begin{lem}\label{lem:many_distals_implies_conjugacy}
Let $(X,T)$ be an infinite minimal subshift of topological rank $K$ and $J$ an index set of cardinality $\#J > K(144K^7)^K$. Suppose that for every $j \in J$ there exists a distal symbolic factor $\pi_j\colon(X,T)\to(Y_j,T)$. Then, there are $i\not=j \in J$ such that $(Y_i,T)$ is conjugate to $(Y_j,T)$.
\end{lem}
\begin{proof}
We start by introducing the necessary objects for the proof and doing some general observations about them. First, thanks to Theorem \ref{theo:factores_k_1}, we know that $\pi_j$ is almost $k_j$-to-1 for some $k_j \leq K$, so, by the Pigeon Principle, there exist $J_1 \subseteq J$ and $k \leq K$ such that $\#J_1 \geq \#J/K > (144K^7)^K$ and $k_j = k$ for every $j \in J_1$. For $j \in J_1$, we fix $z^j = (z_1^j, \dots, z_k^j) \in R_{\pi_j}^k$ with $z_n^j \not= z_m^j$ for all $n\not=m$. Let $Z_j = \overline{\mathrm{orb}}_{T^{[k]}}(z^j)$ and $\rho\colon X^k\to X$ be the factor map that projects onto the first coordinate. By Lemma \ref{lem:kto1_to_groupextension}, $\pi_j$ is $k$-to-1 and $Z_j$ minimal. This imply that if $x = (x_1,\dots,x_k) \in Z_j$, then 
\begin{align}
\{x_1,\dots,x_k\} &= \pi_j^{-1}(\pi_j(x_n)) \text{ for all $n\in \{1,\dots,k\}$,}
\label{eq:lem:many_distals_implies_conjugacy:1} \\
x_n &\not= x_m \text{ for all $n,m \in \{1,\dots,k\}$.}
\label{eq:lem:many_distals_implies_conjugacy:2}
\end{align}
Indeed, since $Z_j$ is minimal, $(T^{[k]})^{n_\ell}z \to x$ for some sequence $(n_\ell)_\ell$, so, 
$$\inf_{n\not=m}\mathrm{dist}(x_n,x_m) \geq \inf_{n\not=m, l\in\Z}\mathrm{dist}(T^lz_n,T^lz_m) > 0,$$
where in the last step is due the fact that $\pi_j$ is distal. This gives \eqref{eq:lem:many_distals_implies_conjugacy:2}. For \eqref{eq:lem:many_distals_implies_conjugacy:1} we first note that $\{x_1,\dots,x_k\} \subseteq \pi_j^{-1}(\pi_j(x_n))$ as $x \in R_{\pi_j}$, and then that the equality must hold since $\#\pi_j^{-1}(\pi_j(x_n)) = k = \#\{x_1,\dots,x_k\}$ by \eqref{eq:lem:many_distals_implies_conjugacy:2}.

The next step is to prove that asymptotic pairs in $Z_j$ are well-behaved:
\begin{claim}
Let $j \in J_1$ and $(x^j = (x_1^j,\dots,x_k^j)$, $\tilde{x}^j = (\tilde{x}_1^j,\dots,\tilde{x}_k^j))$ be a right asymptotic pair in $Z_j$, this is, 
\begin{equation}\label{eq:lem:many_distals_implies_conjugacy:3}
\text{$\lim_{n\to-\infty}\mathrm{dist}((T^{[k]})^nx^j, T^{[k]}\tilde{x}^j) = 0$ and $x^j \not= \tilde{x}^j$.}
\end{equation}
Then, $(x_n^j,\tilde{x}_n^j)$ is right asymptotic for every $n \in \{1,\dots,k\}$.
\end{claim}
\begin{claimproof}
Suppose, with the aim to obtain a contradiction, that $(x_n^j, \tilde{x}_n^j)$ is not right asymptotic for some $n \in \{1,\dots,k\}$. Observe that \eqref{eq:lem:many_distals_implies_conjugacy:3} implies that
\begin{equation}\label{eq:lem:many_distals_implies_conjugacy:4}
\text{for every $m \in \{1,\dots,k\}$, either $(x_m^j,\tilde{x}_m^j)$ is right asymptotic or $x_n^j = \tilde{x}_n^j$.}
\end{equation}
Therefore, $x_n^j = \tilde{x}_n^j$. Using this and that $x^j,\tilde{x}^j \in R_{\pi_j}^k$ we can compute:
\begin{equation*}
\text{$\pi_j(x_m^j) = \pi_j(x_n^j) = \pi_j(\tilde{x}_n^j) = \pi_j(\tilde{x}_l^j)$ for all $m,l\in\{1,\dots,k\}$,}
\end{equation*}
and thus, by \eqref{eq:lem:many_distals_implies_conjugacy:1},
\begin{equation*}
\{x_1^j,\dots,x_k^j\} = \pi_j^{-1}(\pi_j(x_n^j)) = \pi_j^{-1}(\pi_j(\tilde{x}_n^j)) = \{\tilde{x}_1^j,\dots,\tilde{x}_k^j\}.
\end{equation*}
The last equation, \eqref{eq:lem:many_distals_implies_conjugacy:2} and that $x^j\not=\tilde{x}^j$ imply that there exist $m\not=l \in \{1,\dots,k\}$ such that $\tilde{x}_{l}^j = x_m^j$. This last equality and \eqref{eq:lem:many_distals_implies_conjugacy:4} tell us that $x^j_m$ and $x^j_l$ are either asymptotic or equal. But in both cases a contradiction occurs: in the first one with the distality of $\pi$ and in the second one with equation \eqref{eq:lem:many_distals_implies_conjugacy:2}.
\end{claimproof}

Let $j \in J_1$. Since $Y_j$ is infinite, $Z_j$ is a infinite subshift. It is a well-known fact from symbolic dynamics that this implies that there exists a right asymptotic pair $(x^j = (x_1^j,\dots,x_k^j)$, $\tilde{x}^j = (\tilde{x}_1^j,\dots,\tilde{x}_k^j))$ in $Z_j$. We are now going to use Theorem \ref{prop:asymptotic_from_EM} to prove the following:
\begin{claim}
There exists $i,j \in J_1$, $i\not=j$, such that $Z_i = Z_j$.
\end{claim}
\begin{claimproof}
On one hand, by the previous claim, $(x_n^j, \tilde{x}_n^j) \in X^2$ is right asymptotic for every $n \in \{1,\dots, k\}$ and $j \in J_1$. Let $p_n^j \in \Z$ be such that $(T^{p_n^j}x_n^j, T^{p_n^j}\tilde{x}_n^j)$ is centered right asymptotic. On the other hand, Theorem \ref{prop:asymptotic_from_EM} asserts that the set
\begin{equation*}
\{ x_{(0,\infty)} : (x,\tilde{x})
\text{ is centered right asymptotic in }X\}
\end{equation*}
has at most $144K^7$ elements. Since $\#J_1 > (144K^7)^K$, we conclude, by the Pigeonhole principle, that there exist $i,j \in J_1$, $i\not=j$, such that 
\begin{equation}\label{eq:lem:many_distals_implies_conjugacy:5}
\text{$T^{p_n^i}x_n^i$ and $T^{p_n^j}x_n^j$ agree on $(0,\infty)$ for every $n \in \{1,\dots,k\}$.}
\end{equation}
We are going to show that $Z_i = Z_j$. 

Using \eqref{eq:lem:many_distals_implies_conjugacy:5}, we can find $u \in E(X,T)$ such that $uT^{p_n^i}x_n^i = uT^{p_n^j}x_n^j$ for every $n$. 
Then, by putting $y_n^i = ux_n^i$, $y_n^j = ux_n^j$ and $q_n = p_n^j-p_n^i$, we have
\begin{equation*}
\text{$y^i \coloneqq (y_1^i,\dots,y_k^i) \in Z_i$, $y^j \coloneqq (y_1^j,\dots,y_k^j) \in Z_j$ and $y_n^i = T^{q_n}y_n^j$.}
\end{equation*}
Hence, $\pi(y_n^i) = T^{q_n}\pi(y_n^j)$ and Lemma \ref{lem:fabien} can be applied to deduce that $q \coloneqq q_n$ has the same value for every $n$. We conclude that $y^i = T^qy^j \in T^qZ_j = Z_j$, that $Z_i \cap Z_j$ is not empty and, therefore, that $Z_i = Z_j$ as these are minimal systems.
\end{claimproof}

We can now finish the proof. Let $i\not=j\in J_1$ be the elements given by the previous claim, so that $Z \coloneqq Z_i = Z_j$. Let $y \in Y_i$ and $x = (x_1,\dots,x_k) \in \rho^{-1}\pi_i^{-1}(y)\cap Z$. Then, by \eqref{eq:lem:many_distals_implies_conjugacy:1}, $\pi_i^{-1}(y) = \{x_1,\dots,x_k\} = \pi_j^{-1}(\pi_j(x_1))$, and so $\pi_j\pi_i^{-1}(y)$ contains exactly one element, which is $\pi_j(x_1)$. We define $\psi\colon Y_i\to Y_j$ by $\psi(y) = \pi_j(x_1)$. 

Observe that $\pi_i^{-1}\colon Y_i\to 2^X$ is continuous (as $\pi_i$ is distal, hence open) and commutes with $T$. Being $\pi_j$ a factor map, $\psi$ is continuous and commutes with $T$. Therefore, $\psi\colon(Y_i,T)\to(Y_j,T)$ is a factor map. A similar construction gives a factor  map $\phi\colon Y_j\to Y_i$ which is the inverse function of $\psi$. We conclude that $\psi$ is a conjugacy and, thus, that $Y_i$ and $Y_j$ are conjugate.
\end{proof}

\subsection{Non-distal factor maps and asymptotic pairs lying in fibers}\label{subsec:number_factors:non-distal_case}

To deal with non-factor maps, we study asymptotic pairs belonging to fibers of this kind of factors. The starting point is the following lemma.
\begin{lem}\label{lem:either_distal_or_asymptotic}
Let $\pi\colon(X,T) \to (Y,T)$ be a factor between minimal subshifts. Then, either $\pi$ is distal or there exists a fiber $\pi^{-1}(y)$ containing a pair of right or left asymptotic points.
\end{lem}
\begin{proof}
Assume that $\pi$ is not distal. Then, we can find a fiber $\pi^{-1}(y)$ and proximal points $x,x' \in \pi^{-1}(y)$, with $x\not=x'$. 
This implies that for every $k \in \N$ there exist a (maybe infinite) interval $I_k = (a_k,b_k) \subseteq \Z$, with $b_k-a_k\geq k$, for which $x$ and $x'$ coincide on $I$ and $I_k$ is maximal (with respect to the inclusion) with this property. Since $x\not=x'$, then $a_k > -\infty$ or $b_k < \infty$. Hence, there exists an infinite set $E \subseteq \N$ such that $a_k > -\infty$ for every $k \in E$ or $b_k < \infty$ for every $k \in E$. In the first case, we have that $(T^{b_k}(x,x'))_{k\in E}$ has a left asymptotic pair $(z,z')$ as an accumulation point, while in the second case it is a right asymptotic pair $(z,z')$ who is an accumulation point of $(T^{a_k}(x,x'))_{k\in E}$. In both cases we have that $(z,z') \in R_\pi^2$ since $(T^{b_k}(x,x'))_{k\in E}$ and $(T^{a_k}(x,x'))_{k\in E}$ are contained in $R_\pi^2$ and $R_\pi^2$ is closed. Therefore, the fiber $\pi^{-1}(\pi(z))$ contains a pair $z,z'$ of asymptotic points.
\end{proof}

The next lemma allows us to pass from morphisms $\sigma\colon X\to Y$ to factors $\pi\colon X'\to Y$ in such a way that $X'$ is defined on the same alphabet as $X$ and has the ``same'' asymptotic pairs. We remark that its proof is simple, but tedious. 
\begin{lem}\label{lem:inflated_morphism_sequence}
Let $X \subseteq \cA^+$ be an aperiodic subshift, $\sigma\colon\cA^+\to\cB^+$ be a morphism and $Y = \bigcup_{k\in\Z}T^k\sigma(X)$. Define the morphism $i_\sigma\colon\cA^+\to\cA^+$ by $i_\sigma(a) = a^{|\sigma(a)|}$, $a \in \cA$, and $X' = \bigcup_{k\in\Z}T^ki_\sigma(X)$. Then, centered asymptotic pairs in $X'$ are of the form $(i_\sigma(x),i_\sigma(\tilde{x}))$, where $(x,\tilde{x})$ is a centered asymptotic pair in $X$, and there exists a factor map $\pi\colon(X',T)\to(Y,T)$ such that $\pi(i_\sigma(x)) = \tau(x)$ for all $x \in X$.
\end{lem}
\begin{proof}
Our first objective is to prove that $(X,i_\sigma)$ is recognizable. We start by observing that
\begin{equation}\label{eq:lem:inflated_morphism_sequence:1}
\text{if $(k,x)$, $(\tilde{k},\tilde{x})$ are centered $i_\sigma$-factorizations of $y \in X'$, then $x_0 = \tilde{x}_0$.}
\end{equation}
Indeed, since the factorization are centered, we have $x_0 = i_\sigma(x_0)_k = y_0 = i_\sigma(\tilde{x}_0)_{\tilde{k}} = \tilde{x}_0$. 

Let $\Lambda$ be the set of tuples $(k,x,\tilde{k},\tilde{x})$ such that $(k,x),(\tilde{k},\tilde{x})$ are centered $i_\sigma$-factorizations of the same point. Moreover, for $\cR \in \{=,>\}$, let $\Lambda_\cR$ be the set of those $(k,x,\tilde{k},\tilde{x}) \in \Lambda$ satisfying $k\ \cR\ \tilde{k}$.
\begin{claim}
If $(k,x,\tilde{k},\tilde{x}) \in \Lambda_=$, then $(0,Tx,0,T\tilde{x}) \in \Lambda_=$, and if $(k,x,\tilde{k},\tilde{x}) \in \Lambda_>$, then $(|i_\sigma(x_0)|-k+\tilde{k},\tilde{x},0,Tx) \in \Lambda_>$.
\end{claim}
\begin{claimproof}
If $(k,x,\tilde{k},\tilde{x}) \in \Lambda_=$, then, since $x_0 = \tilde{x}_0$ by \eqref{eq:lem:inflated_morphism_sequence:1}, we can write $i_\sigma(Tx) = T^ki_\sigma(x) = T^{\tilde{k}}i_\sigma(\tilde{x}) = i_\sigma(T\tilde{x})$. Thus, $(0,Tx,0,T\tilde{x}) \in \Lambda_=$. Let now $(k,x,\tilde{k},\tilde{x}) \in \Lambda_>$ and $y \coloneqq T^ki_\sigma(x) = T^{\tilde{k}}i_\sigma(\tilde{x})$. We note that
\begin{equation*}
T^{|i_\sigma(x_0)|-k+\tilde{k}}i_\sigma(\tilde{x}) =
T^{|i_\sigma(x_0)|-k}y =
T^{|i_\sigma(x_0)|}i_\sigma(x) =
i_\sigma(Tx),
\end{equation*}
so $(|i_\sigma(x_0)|-k+\tilde{k},\tilde{x})$ and $(0,Tx)$ are $i_\sigma$-factorization of the same point. Now, since $x_0 = \tilde{x}_0$ (by \eqref{eq:lem:inflated_morphism_sequence:1}) and $(k,x),(\tilde{k},\tilde{x})$ are centered, we have $k,\tilde{k}\in[0,|i_\sigma(x_0)|)$. This and and the fact that $k>\tilde{k}$ imply that $k-\tilde{k}\in(0,|i_\sigma(x_0)|)$. Therefore, $|i_\sigma(x_0)|-k+\tilde{k} \in (0,|i_\sigma(x_0)|)$ and, consequently, $(|i_\sigma(x_0)|-k+\tilde{k},\tilde{x},0,Tx) \in \Lambda_>$.
\end{claimproof}

We prove now that $(X,i_\sigma)$ is recognizable. Let $(k,x,\tilde{k},\tilde{x}) \in \Lambda$. We have to show that $(k,x) = (\tilde{k},\tilde{x})$. 
First, we consider the case in which $k = \tilde{k}$. In this situation, the previous claim implies that $(0,Tx,0,T\tilde{x}) \in \Lambda_=$. We use again the claim, but with $(0,Tx,0,T\tilde{x})$, to obtain that $(0,T^2x,0,T^2\tilde{x}) \in \Lambda_=$. By continuing in this way, we get $(0,T^nx,0,T^n\tilde{x}) \in \Lambda_=$ for any $n\geq0$. Then, \eqref{eq:lem:inflated_morphism_sequence:1} implies that $x_n = \tilde{x}_n$ for all $n\geq0$. A similar argument shows that $x_n = \tilde{x}_n$ for any $n\leq0$, and so $(k,x) = (\tilde{k},\tilde{x})$. 
We now do the case $k > \tilde{k}$. Another application of the claim gives us $(p_1,\tilde{x},0,Tx) \in \Lambda_>$ for some $p_1 \in \Z$. As before, we iterate this procedure to obtain that $(p_2,Tx,0,T\tilde{x}) \in \Lambda_>$, $(p_3,T\tilde{x},0,T^2x) \in \Lambda_>$ and so on. From these relations and \eqref{eq:lem:inflated_morphism_sequence:1} we deduce that $x_0 = \tilde{x}_0$, $\tilde{x}_0 = (Tx)_0 = x_1$, $x_1 = (Tx)_0 = (T\tilde{x})_0 = \tilde{x}_1$, $\tilde{x}_1 = (T\tilde{x})_0 = (T^2x)_0 = x_2$, etc. We conclude that $x_n = \tilde{x}_n = x_0$ for any $n \geq 0$. Then, by compacity, the periodic point $\cdots x_0.x_0x_0\cdots$ belongs to $X$, contrary to our aperiodicity hypothesis on $X$. Thus, the case $k > \tilde{k}$ does not occurs. This proves that $(X,i_\sigma)$ is recognizable.

Using the property we just proved, we can define the factor map $\pi\colon X'\to Y$ as follows: if $x' \in X'$, then we set $\pi(x') = T^k\tau(x) \in Y$, where $(k,x)$ is the unique centered $i_\sigma$-factorization of $x'$ in $X$. To show that $\pi$ is indeed a factor map, we first observe that since 
\begin{equation}\label{eq:lem:inflated_morphism_sequence:2}
\text{$|\tau(a)| = |i_\sigma(a)|$ for all $a \in \cA$,}
\end{equation}
$\pi$ commutes with $T$. Moreover, thanks to (iii) in Remark \ref{rem:factorizations}, $\pi$ is continuous. Finally, if $y \in Y$, then by the definition of $Y$ there exist a centered $(k,x)$ $\tau$-factorization of $y$ in $X$. Thus, by \eqref{eq:lem:inflated_morphism_sequence:2}, $(k,x)$ is a centered $i_\sigma$ factorization of $x' \coloneqq T^ki_\sigma(x)$. Therefore, $\pi(x') = y$ and $\pi$ is onto. Altogether, these arguments show that $\pi$ is a factor map. That $\pi(i_\sigma(x)) = \tau(x)$ for every $x \in X$ follows directly from the definition of $\pi$. 

It left to prove the property about the asymptotic pairs. We only prove it for left asymptotic pairs since the other case is similar. We will use the following notation: if $Z$ is a subshift, then $A(Z)$ denotes the set of centered left asymptotic pairs. To start, we observe that $(i_\sigma(x),i_\sigma(x')) \in A(X')$ for every $(x,\tilde{x}) \in A(X)$. 
Let now $(z,\tilde{z}) \in A(X')$, and $(k,x)$ and $(\tilde{k},\tilde{x})$ be the unique centered $i_\sigma$-factorizations of $z$ and $\tilde{z}$ in $X$, respectively. We have to show that $k = \tilde{k} = 0$ and that $(x,\tilde{x}) \in A(X)$. 
Due to (iii) in Remark \ref{rem:factorizations}, $(X,i_\sigma)$ has a recognizability constant. This and the fact that $(z,\tilde{z})$ is centered left asymptotic imply that $(k,x)$ and $(\tilde{k},\tilde{x})$ have a common cut in $(-\infty,0]$, this is, that there exist $p,q \leq 0$ such that
\begin{equation*}
m \coloneqq 
-k-|i_\sigma(x_{[p,0)})| = 
-\tilde{k}-|i_\sigma(\tilde{x}_{[q,0)})| \in (-\infty, 0].
\end{equation*}
We take $m$ as big as possible with this property. Then, $x_p \not= \tilde{x}_q$. Moreover, being $z_{m} = x_p$ and $\tilde{z}_{m} = \tilde{x}_p$ by the definition of $i_\sigma$, we have that $z_m \not= \tilde{z}_m$ and consequently, by also using that $(z,\tilde{z})$ is centered left asymptotic, that $m \geq 0$. We conclude that $m = 0$, this is, that $k+|i_\sigma(x_{[p,0)})| = \tilde{k}+|i_\sigma(\tilde{x}_{[q,0)})| = 0$. Hence, $k =\tilde{k}=p=q=0$. Now, it is clear that $x_{(-\infty,p]} = \tilde{x}_{(-\infty,q]}$, so from the last equations we obtain that $(x,\tilde{x}) \in A(X)$. This completes the proof.
\end{proof}

We will also need the following lemma to slightly strengthen Proposition \ref{prop:asymptotic_from_EM}.
\begin{lem}\label{lem:bounding_asymptotic_pairs}
Let $X \subseteq \cA^\Z$ be an aperiodic subshift with $L$ asymptotic tails. Then, $(X,T)$ has at most $2L^2\cdot \#\cA^2$ centered asymptotic pairs.
\end{lem}
\begin{proof}
Let $\cP_r$ be the set of centered right asymptotic pairs in $X$ and
$\cT_r = \{x_{(0,\infty)} : (x,\tilde{x}) \in \Lambda\} \subseteq \cA^{\N_{\geq1}}$ be the set of right asymptotic tails, where $\N_{\geq1} = \{1,2,\dots\}$. We are going to prove that
\begin{equation}\label{eq:lem:bounding_asymptotic_pairs:1}
\#\cP_r \leq \#\cT_r^2\cdot \#\cA^2.
\end{equation}
Once this is done, we will have by symmetry the same relation for the centered left asymptotic pairs $\cP_l$, and thus we are going to be able to conclude that the number of centered asymptotic pairs in $X$ is at most $(\#\cT_r^2+\#\cT_l^2)\cdot \#\cA^2 \leq 2L^2\cdot\#\cA^2$, completing the proof.

Let $(x,\tilde{x}) \in \cP_r$ and $\cR_x = \{k\leq0 : x_{(k,\infty)} \in \cT_r\}$. We claim that $\#\cR_x \leq \#\cT_r$. Indeed, if this is not the case, then, by the Pigeonhole principle, we can find $k' < k$ and $w \in \cT_r$ such that $w = x_{(k,\infty)} = x_{(k',\infty)}$. But this implies that $w$ has period $k-k'$, and so $X$ contains a point of period $k-k'$, contrary to the aperiodicity hypothesis. Thus, $\cR_x$ is finite and, since $\cR_x$ is non-empty as it contains $x_{(0,\infty)}$, $k_x \coloneqq \min\cR_x$ is a well-defined non-positive integer.

Let now $\phi\colon \cP_r \to \cT_r^2\times\cA^2$ be the function defined by
\begin{equation*}
\phi(x,\tilde{x}) = 
(x_{(k_x, \infty)}, \tilde{x}_{(k_{\tilde{x}},\infty)}, 
x_{k_x}, \tilde{x}_{k_{\tilde{x}}})
\end{equation*}
If $\phi$ is injective, then \eqref{eq:lem:bounding_asymptotic_pairs:1} follows. Let us then prove that $\phi$ is injective.

We argue by contradiction and assume that there exist $(x,\tilde{x})\not=(y,\tilde{y})$ such that $\phi(x,\tilde{x}) = \phi(y,\tilde{y}) = (z,\tilde{z},a,\tilde{a})$. Without loss of generality, we may assume that $x \not= y$. Then, $x_{(k_x,\infty)} = z = y_{(k_y,\infty)}$ and $x_{k_x} = a = y_{k_y}$. Being $x \not= y$, this implies that $(x,y)$ is asymptotic. Furthermore, it implies that there exist $p < k$ and $q < \ell$ such that $(T^px, T^qy)$ is centered right asymptotic. In particular, $x_{(p,\infty)} \in \cT_r$ and $p < k_x$, contrary to the definition of $k_x$. We conclude that $\phi$ is injective and thereby complete the proof of the lemma.
\end{proof}

\begin{lem}\label{lem:not_distal_implies_recursion}
Let $X \subseteq \cA^\Z$ be a subshift of topological rank $K$, $J$ be an index set and, for $j \in J$, let $\tau_j\colon\cA^+\to\cB_j^+$ be a morphism. Suppose that for every $j \in J$
\begin{enumerate}[label=(\Roman*)]
\item $Y_j = \bigcup_{k\in\Z}T^k\tau_j(X)$ is aperiodic;
\item for every fixed $a \in \cA$, $|\tau_j(a)|$ is equal to a constant $\ell_a$ independent of $j \in J$.
\end{enumerate}
Then, one of the following situations occur:
\begin{enumerate}
\item There exist $i,j\in J$, $i\not=j$, such that $(Y_i,T)$ is conjugate to $(Y_j,T)$.
\item There exist $\phi\colon\cA^+\to\cA_1^+$ with $\#\cA_1 < \#\cA$, a set $J_1 \subseteq J$ having at least $\#J / 2\#\cA^2(144K^7)^2 - K(144K^7)^K$ elements, and morphisms $\tau'_j\colon \cC_1^+\to\cB_j$, $j \in J_1$, such that $\tau_j = \tau'_j\phi$. 
In particular, the hypothesis of this lemma hold for $X_1 \coloneqq \bigcup_{k\in\Z}T^k\phi(X)$ and $\tau'_j$, $j \in J_1$.
\end{enumerate}
\end{lem}
\begin{proof}
Let $\mathfrak{i}\colon\cA^+\to\cA^+$ be the morphism defined by $\mathfrak{i}(a) = a^{\ell_a}$, $a \in \cA$, and $X' = \bigcup_{k\in\Z}T^k\mathfrak{i}(X)$. We use Lemma \ref{lem:inflated_morphism_sequence} with $X$ and $\tau_j$ to obtain a factor map $\pi_j\colon(X',T) \to (Y_j,T)$ such that
\begin{equation}\label{eq:lem:not_distal_implies_recursion:1}
\text{$\pi(\mathfrak{i}(x)) = \tau_j(x)$ for every $x \in X$.}
\end{equation}
If $\pi_j$ is distal for $K(144K^7)^K+1$ different values of $j \in J$, then by Lemma \ref{lem:many_distals_implies_conjugacy} we can find $i,j$ such that $(Y_i,T)$ is conjugate to $(Y_j,T)$. Therefore, we can suppose that there exists $J' \subseteq J$ such that
\begin{equation}\label{eq:lem:not_distal_implies_recursion:2}
\text{$\#J' \geq \#J-K(144K^7)^K$ and $\pi_j$ is not distal for every $j \in J'$.}
\end{equation}
From this and Lemma \ref{lem:either_distal_or_asymptotic} we obtain, for every $j \in J'$, a centered asymptotic pair $(x^{(j)},\tilde{x}^{(j)})$ in $X'$ such that $\pi_j(x^{(j)}) = \pi_j(\tilde{x}^{(j)})$. This and \eqref{eq:lem:not_distal_implies_recursion:1} imply that
\begin{equation}\label{eq:lem:not_distal_implies_recursion:3}
\tau_j(x^{(j)}) = \pi_j(x^{(j)}) = \pi_j(\tilde{x}^{(j)}) = \tau_j(\tilde{x}^{(j)}).
\end{equation}
Now, by Lemma \ref{lem:bounding_asymptotic_pairs}, $X$ has at most $2\#\cA^2(144K^7)^2$ centered asymptotic pairs and thus, thanks to Lemma \ref{lem:inflated_morphism_sequence}, the same bound holds for $X'$. Therefore, by the Pigeonhole principle, there exist $J_1 \subseteq J$ satisfying $\#J_1 \geq \#J'/2\#\cA^2(144K^7)^2 \geq \#J / 2\#\cA^2(144K^7)^2 - K(144K^7)^K$ and a centered asymptotic pair $(x,\tilde{x})$ in $X'$ such that $(x,\tilde{x}) = (x^{(j)},\tilde{x}^{(j)})$ for every $j \in J_1$. 

We assume that $(x,\tilde{x})$ is right asymptotic as the other case is similar. Then, equation \eqref{eq:lem:not_distal_implies_recursion:3} implies that if $\ell = \sum_{a\in\cA}\ell_a$, then, for every $j \in J_1$,
\begin{equation}\label{eq:lem:not_distal_implies_recursion:4}
\text{one of the words in $\{\tau_j(x_{[0,\ell)}), \tau_j(\tilde{x}_{[0,\ell)})\}$ is a prefix of the other.}
\end{equation}
This, hypothesis (II) and the fact that, since $(x,\tilde{x})$ a centered asymptotic pair, $x_0 \not= \tilde{x}_0$ allow us to use Lemma \ref{lem:lower_free_rank_s_equal_0} with $u \coloneqq x_{[0,\ell)}$, $v \coloneqq \tilde{x}_{[0,\ell)}$, $J \coloneqq J_1$ and $w^j \coloneqq \tau_j(x_{[0,\infty)})_{[0,\ell)}$ and obtain morphisms $\phi\colon \cA^+\to\cA_1^+$ and $\tau'_j\colon\cA_1^+\to\cB_j^+$, $j \in J_1$, such that $\#\cA_1 < \#\cA$, $\tau_j = \tau'_j\phi$ and
\begin{equation}\label{eq:lem:not_distal_implies_recursion:5}
\text{for every $a \in \cA_1$, $\ell'_a \coloneqq |\tau'_j(c)|$ does not depend on the chosen $j \in J$.}
\end{equation}

Finally, we observe that $X_1$ and $\tau'_j$, $j \in J_1$, satisfy the hypothesis of the lemma: condition (I) holds since, by the relation $\tau_j = \tau'_j\phi$, the subshift $X_1 \coloneqq \bigcup_{k\in\Z}T^k\phi(X)$ satisfies that $\bigcup_{k\in\Z}T^k\tau'_j(X_1) = Y_j$ is aperiodic; condition (II) is given by \eqref{eq:lem:not_distal_implies_recursion:5}.
\end{proof}

\subsection{Proof of main result}\label{subsec:number_factors:main_result}

We now prove Theorem \ref{theo:intro:4}. We restate it for convenience.

\begin{ctheo}{\ref{theo:intro:4}}
Let $(X,T)$ be an minimal subshift of topological rank $K$. Then, $(X,T)$ has at most $(3K)^{32K}$ aperiodic symbolic factors up to conjugacy.
\end{ctheo}
\begin{proof}
We set $R = (3K)^{32K}$. We prove the theorem by contraction: assume that there exist $X \subseteq \cA^\Z$ of topological rank $K$ and, for $j \in \{0,\dots,R\}$, factor maps $\pi_j\colon(X,T) \to (Y_j,T)$ such that $(Y_i,T)$ is not conjugate to $(Y_j,T)$ for every $i\not=j \in \{0,\dots,R\}$.
We remark that $X$ must be infinite as, otherwise, it would not have any aperiodic factor.

To start, we build $\cS$-representations for the subshifts $X$ and $Y_j$. Let $\boldsymbol{\sigma} = (\sigma_n\colon\cA_{n+1}^+\to\cA_n^+)_{n\in\N}$ be the primitive and proper directive sequence of alphabet rank $K$ generating $X$ given by Theorem \ref{theo:FTR_from_DDPM}. Let $r \in \N$ be such that every $\pi_j$ has a radius $r$ and let $\cB_j$ the alphabet of $Y_j$. By contracting $\boldsymbol{\sigma}$, we can assume that $\sigma_0$ is $r$-proper and $\#\cA_n = K$ for all $n\geq 1$. Then, we can use Lemma \ref{lem:factor_to_morphism} to find morphisms $\tau_{j} \colon \cA_1^+\to\cB_j^+$ such that
\begin{equation}\label{eq:theo:intro:4:1}
\text{$\pi_j(\sigma_1(x)) = \tau_j(x)$ for all $x \in X_{\boldsymbol{\sigma}}^{(1)}$ and $|\tau_{j}(a)| = |\sigma_0(a)|$ for all $a \in \cA_1$.}
\end{equation}

Next, we inductively define subshifts $X_n \subseteq \cC_n^\Z$ and morphisms $\{\tau_{n,j}\colon\cC_n^+\to\cB_j : j\in J_n\}$ such that
\begin{enumerate}
\item[(i)] $X_n$ has topological rank at most $K$;
\item[(ii)] $Y_j = \bigcup_{k\in\Z} \tau_{n,j}(X_n)$;
\item[(iii)] for every $c \in \cC_n$, $\ell_{n,a} \coloneqq |\tau_{n,j}(c)|$ does not depend on the chosen $j \in J_n$.
\end{enumerate}
First, we set $X_0 = X_{\boldsymbol{\sigma}}^{(1)}$, $\cC_0 = \cA_1$, $J_0 = J$ and, for $j\in J_0$, $\tau_{0,j} = \tau_{j}$, and note that by the hypothesis and \eqref{eq:theo:intro:4:1}, they satisfy (i), (ii) and (iii). 
Let now $n \geq 0$ and suppose that $X_n \subseteq \cC_n^\Z$ and $\tau_{n,j}$, $j\in J_n$, has been defined in a way such that (i), (ii) and (iii) hold.
If $\#J_n / 2\#\cA^2(144K^7)^2 - K(144K^7)^K \leq 1$, then the procedure stops. 
Otherwise, we define step $n+1$ as follows. 
Thanks to (i), (ii), (iii) we can use Lemma \ref{lem:not_distal_implies_recursion}, and since there are no two conjugate $(Y_i,T)$, this lemma gives us a morphism $\phi\colon\cC_n^+\to\cC_{n+1}^+$, a set $J_{n+1} \subseteq J_n$ and morphisms $\{\tau_{n+1,j}\colon\cC_{n+1}^+\to\cB_j^+ : j\in J_{n+1}\}$ such that
\begin{equation*}
\text{$\#\cC_{n+1} < \#\cC_n$, $\#J_{n+1} \geq \#J_n / 2\#\cC_n^2(144K^7)^2 - K(144K^7)^K$ and $\tau_{n,j} = \tau_{n+1,j}\phi_n$. }
\end{equation*}
Furthermore, $X_{n+1} \coloneqq \bigcup_{k\in\Z}T^k\phi_n(X_n)$ and $\tau_{n+1,j}$ satisfy the hypothesis of that lemma, this is, conditions (ii) and (iii) above. 
Since $(\phi_n\dots\phi_0\sigma_1,\sigma_2,\sigma_3,\dots)$ is a primitive and proper sequence of alphabet rank $K$ generating $X_{n+1}$, Theorem \ref{theo:rank_of_factor} implies that condition (i) is met as well.

Since $\#\cC_0 > \#\cC_1 > \dots$, there is a last $\cC_N$ defined. Our next objective is to prove that $N \geq K$. Observe that $\#\cC_n \leq K$, so
\begin{equation*}
\text{$\#J_{n+1} \geq \#J_n / 2K^2(144K^7)^2 - K(144K^7)^K$ for any $n \in \{0,\dots,N-1\}$.}
\end{equation*}
Using this recurrence and the inequalities $\#J_0 > (3K)^{32K}$ and $K\geq2$, it is routine to verify that the following bound holds that for every $n \in \{0,\dots,K-1\}$ such that the $n$th step is defined:
\begin{equation*}
\#J_n / 2\#\cC_n^2(144K^7)^2 - K(144K^7)^K > 1
\end{equation*}
Therefore, $N \geq K$. We conclude that $\#\cC_N \leq \#\cC_0-K = 0$, which is a contradiction.
\end{proof}

\begin{rem}
In Theorem 1 of \cite{durand_2000}, the author proved that linearly recurrent subshifts have finitely many aperiodic symbolic factors up to conjugacy. Since this kind of systems have finite topological rank (see Remark \ref{rem:coalescency}), Theorem \ref{theo:intro:4} generalizes the theorem of \cite{durand_2000} to the much larger class of minimal finite topological rank subshifts.
\end{rem}

\section{Appendix}

To prove Proposition \ref{prop:make_sequence_reco}, we start with some lemmas concerning how to construct recognizable pairs $(Z,\tau)$ for a fixed subshift $Y = \bigcup_{k\in\Z}T^k\tau(Z)$.

\subsection{Codings of subshifts}\label{sec:codings_of_subshifts}

If $Y \subseteq \cB^\Z$ is a subshift, $U \subseteq Y$ and $y\in Y$, we denote by $\cR_U(y)$ the set of {\em return times} of $y$ to $U$, this is, $\cR_U(y) = \{k\in\Z : T^ky \in U\}$. We recall that the set $C_\tau(k,z)$ in the lemma below corresponds to the cuts of $(k,z)$ (see Definition \ref{defi:factorizations} for further details).
\begin{lem}\label{lem:build_reco_single}
Let $Y \subseteq \cB^\Z$ be an aperiodic subshift, with $\cB \subseteq \cL(Y)$. Suppose that $U \subseteq Y$ is
\begin{enumerate}[label=(\Roman*)]
	\item $d$-syndetic: for every $y \in Y$ there exists $k \in [0,d-1]$ with $T^ky\in U$,
	\item of radius $r$: $U \subseteq \bigcup_{u\in\cA^r,v\in\cA^{r+1}}[u.v]$,
	\item $\ell$-proper: $U \subseteq [u.v]$ for some $u,v\in\cA^\ell$,
	\item $\rho$-separated: $U,TU,\dots,T^{\rho-1}U$ are disjoint.
\end{enumerate}
Then, there exist a letter-onto morphism $\tau\colon \cC^+ \to \cB^+$ and a subshift $Z \subseteq \cC^\Z$ such that
\begin{enumerate}
\item \label{lem:build_reco_single:image}
$Y = \bigcup_{n\in\Z}T^n\tau(Z)$ and $\cC \subseteq \cL(Y)$,
\item \label{lem:build_reco_single:reco}
$(Z,\tau)$ is recognizable with constant $r+d$,
\item \label{lem:build_reco_single:lengths}
$|\tau| \leq d$, $\langle\tau\rangle \geq \rho$ and $\tau$ is $\min(\rho,\ell)$-proper,
\item \label{lem:build_reco_single:cuts}
$C_\tau(k,z) = \cR_U(y)$ for all $y \in Y$ and $\tau$-factorization $(k,z)$ of $y$ in $Z$.
\end{enumerate}
\end{lem}
\begin{rem}\label{rem:proper_implies_separated}
If $U \subseteq Y$ satisfies (III), then $U$ is $\rho \coloneqq \min\per(\cL_{\ell}(Y))$-separated. Indeed, if $U\cap T^k U\not=\emptyset$ for some $k > 0$, then $[v]\cap T^k[v]\not=\emptyset$, where $v \in \cA^\ell$ is such that $U \subseteq [v]$. Hence, $v$ is $k$ periodic and $k \geq \rho$. 
\end{rem}
\begin{proof}
Let $y \in Y$. By (I), the sets $\cR_U(y)\cap[0,\infty)$, $\cR_U(y)\cap(-\infty,0]$ are infinite. Thus, we can write $\cR_U(y) = \{\dots k_{-1}(y) < k_0(y) < k_1(y) \dots \}$, with $\min\{i\in\Z:k_i(y) > 0\}=1$. Let $\cW = \{y_{[k_i(y), k_{i+1}(y))} : y\in Y,\ i\in\Z\} \subseteq \cB^+$. By (I), $\cW$ is finite, so we can write $\cC \coloneqq \{1,\dots, \#\cW\}$ and choose a bijection $\phi\colon \cC \to \cW$. Then, $\phi$ extends to a morphism $\tau\colon \cC^+ \to \cB^+$. As $\cB \subseteq \cL(Y)$, $\phi$ is letter-onto. We define $\psi\colon Y \to \cC^\Z$ by $\psi(y) = (\phi^{-1}(y_{[k_i(y), k_{i+1}(y))}))_{i\in\Z}$ and set $Z = \psi(Y)$. We are going to prove that $\tau$ and $Z$ satisfy (1-4). 

\begin{claim}\
\begin{enumerate}[label=(\roman*)]
	\item If $y_{[-d-r,d+r]} = y'_{[-d-r,d+r]}$, then $\psi(y)_0 = \psi(y')_0$,
	\item $\tau(\psi(y)) = T^{k_0(y)}y$,
	\item $T^j\psi(y) = \psi(T^ky)$ for $j\in \Z$ and $k \in [k_j(y),k_{j+1}(y))$.
\end{enumerate}
\end{claim}
\begin{claimproof}
Let $y,y' \in Y$ such that $y_{[-d-r,d+r]} = y'_{[-d-r,d+r]}$.
By (I), we have $k_{i+1}(y)-k_i(y) \leq d$ for all $i \in \Z$ and, thus, $|k_0(y)|, |k_1(y)| \leq d$. Since $U$ has radius $r$ and $y_{[-d-r,d+r]} = y'_{[-d-r,d+r]}$, we deduce that $k_0(y) = k_0(y')$ and $k_1(y) = k_0(y')$. Hence, $\psi(y)_0 = \phi^{-1}(y_{[k_0(y), k_1(y))}) =  \phi^{-1}(y'_{[k_0(y'), k_1(y'))}) =\psi(y')_0$. To prove (ii) we compute:
\begin{align*}
	\tau(\psi(y)) &=
	\tau(\cdots \phi^{-1}(y_{[k_{-1}(y),k_0(y))}).\phi^{-1}(y_{[k_0(y),k_1(y))})\cdots) \\
	&= \cdots y_{[k_{-1}(y),k_0(y))}.y_{[k_0(y),k_1(y))}\cdots = T^{k_0}y.
\end{align*}
Finally, for (iii) we write, for $k \in [k_j(y),k_{j+1}(y))$,
\begin{align*}
	T^j\psi(y) =
	\dots \phi^{-1}(y_{[k_{j-1}(y),k_j(y))}).\phi^{-1}(y_{[k_j(y),k_{j+1}(y))})\dots
	= \psi(T^ky).
\end{align*}
\end{claimproof}

Now we prove the desired properties of $\tau$ and $Z$.
\begin{enumerate}[leftmargin=0pt, itemindent=25pt]
\item From (i), we see that $\psi$ is continuous and, therefore, $Z$ is closed. By (iii), $Z$ is also shift-invariant and, then, a subshift. By (ii), $Y = \bigcup_{n\in\Z}T^n\tau(Z)$. The condition $\cC \subseteq \cL(Y)$ follows from the definition of $\cW$ and $\tau$.

\item We claim that the only centered $\tau$-interpretation in $Z$ of a point $y \in Y$ is $(-k_0(y), \psi(y))$. Indeed, this pair is a $\tau$-interpretation in $Z$ by (ii), and it is centered because $k_0(y) \leq 0 < k_1(y)$ implies $-k_0(y) \in [0, k_1(y)-k_0(y)) = [0, |\psi(y)_0|)$. Let $(n,z)$ be another centered $\tau$-interpretation of $y$ in $Z$. By the definition of $Z$, there exists $y' \in Y$ with $z = \psi(y')$. Then, by (ii), 
\begin{equation}\label{eq:lem:build_reco_single:1}
T^{n+k_0(y')}y' = T^n\tau(\psi(y')) = T^n\tau(z) = y.
\end{equation}
Now, on one hand, we have $|\tau(z_0)| = |\tau(\psi(y')_0)| = k_1(y')-k_0(y')$. On the other hand, that $(n,\psi(y'))$ is centered gives that $n \in [0,|\tau(z_0)|)$. Therefore, $n +k_0(y')\in (k_0(y'),k_1(y')]$. We conclude from this, (iii) and \eqref{eq:lem:build_reco_single:1} that $\psi(y') = \psi(y)$. Hence, $y = T^n\tau\psi(y') = T^n\tau\psi(y) = T^{n+k_0(y)}y$, which implies that $n = -k_0(y)$ as $Y$ is aperiodic. This proves that $(-k_0(y), \psi(y))$ is the only $\tau$-interpretation of $y$ in $Z$. From this and (i) we deduce property (2).

\item Since $U$ is $d$-syndetic, $|\tau(\psi(y)_i)| = |y_{[k_i(y), k_{i+1}(y))}| = k_{i+1}(y)-k_i(y) \leq d$ for $y\in Y$ and $i\in\Z$, so $|\tau| \leq d$. Similarly, we can obtain $\langle\tau\rangle \geq \rho$ using that $U$ is $\rho$-separated. Let $u,v \in \cB^\ell$ satisfying $U \subseteq [u.v]$. Since $k_i, k_{i+1} \in \cR_U(y)$, we have that $u = y_{[k_i(y), k_i(y)+|u|)}$, $v = y_{[k_{i+1}(y)-|v|, k_{i+1}(y))}$ and, thus, that $\tau$ is $\min(\ell,\langle\tau\rangle)$-proper. In particular, it is $\min(\ell,\rho)$-proper.

\item This follows directly from the definition of $\tau$ and $\cR_U(y)$.
\end{enumerate}
\end{proof}

\begin{lem}\label{lem:make_commutative}
For $j\in \{0,1\}$, let $\sigma_j \colon \cA_j^+ \to \cB^+$ be a morphism and $X_j \subseteq \cA_j^\Z$ be a subshift such that $Y \coloneqq \bigcup_{n\in\Z}T^n\sigma_j(X_j)$ and $\cA_j \subseteq \cL(X_j)$ for every $j \in \{0,1\}$. Suppose that:
\begin{enumerate}
	\item $(X_0, \sigma_0)$ is recognizable with constant $\ell$,
	\item $\sigma_1$ is $\ell$-proper,
	\item $C_{\sigma_0}(k^0,x^0)(y) \supseteq C_{\sigma_1}(k^1,x^1)(y)$ for all $y \in Y$ and $\sigma_j$-factorizations $(k^j,x^j)$ of $y$ in $X_j$, $j=0,1$.
\end{enumerate}
Then, there exist a letter-onto and proper morphism $\nu\colon\cA_1^+\to\cA_0^+$ such that $\sigma_1 = \sigma_0\nu$ and $X_0 = \bigcup_{k\in\Z}T^k\nu(X_1)$.
\end{lem}
\begin{proof}
Since $\sigma_1$ is $\ell$-proper, we can find $u,v \in \cB^\ell$ such that $\sigma_1(a)$ starts with $u$ and ends with $v$ for every $a \in \cA_1$. We define $\nu$ as follows. Let $a\in\cA_1$ and $x \in X_1$ such that $a = x_0$. Since $\sigma_1$ is $\ell$-proper, the word $v.\sigma_1(a)u$ occurs in $\sigma_1(x) \in Y$ at position $0$. By (3), we can find $w \in \cL(X_0)$ with $\sigma_1(x_0) = \sigma_0(w)$. We set $\nu(a) = w$. Since $(X_0,\sigma_0)$ is recognizable with constant $\ell$ and $u,v$ have length $\ell$, $w$ uniquely determined by $v.\sigma_1(a)u$ and, therefore, $\nu$ is well defined. Moreover, the recognizability implies that the first letter of $\nu(a)$ depends only on $v.u$, so $\nu$ is left-proper. A symmetric argument shows that $\nu$ is right-proper and, in conclusion, that it is proper. We also note that $\nu$ is letter-onto as $\cA_0 \subseteq \cL(X_0)$. It follows from the definition of $\nu$ that $\sigma_1 = \sigma_0\nu$. Now, let $x \in X_1$ and $(k,x')$ be a centered $\sigma_0$-factorization of $\sigma_1(x)$ in $X_0$. By (3), $k=0$ and $\sigma_1(x_j) = \sigma_0(x'_{[k_j,k_{j+1})})$ for some sequence $... < k_{-1} < k_0 < ...$ Hence, by the definition of $\nu$, $\nu(x) = x' \in X_0$. This argument shows that $X'_0 \coloneqq \bigcup_{n\in\Z}T^n\nu(X_1) \subseteq X_0$. Then, $\bigcup_{n\in\Z}T^n\sigma_0(X'_0) = \bigcup_{n\in\Z}T^n\sigma_0\nu(X_1) = Y$, where in the last step we used that $\sigma_0\nu=\sigma_1$. Since the points in $Y$ have exactly one $\sigma_0$-factorization, we must have $X'_0 = X_0$. This ends the proof. 
\end{proof}

\subsection{Factors of $\cS$-adic sequences}

Now we are ready to prove Proposition \ref{prop:make_sequence_reco}. For convenience, we repeat its statement.
\begin{prop}
Let $\boldsymbol{\sigma} = (\sigma_n\colon \cA_n \to \cA_{n-1})_{n\geq0}$ be a letter-onto, everywhere growing and proper directive sequence. Suppose that $X_{\boldsymbol{\sigma}}$ is aperiodic. Then, there exists a contraction $\boldsymbol{\sigma}' = (\sigma_{n_k})_{k\in\N}$ and a letter-onto and proper factor $\boldsymbol{\phi}\colon \boldsymbol{\sigma}' \to \boldsymbol{\tau}$, where $\boldsymbol{\tau}$ is letter-onto, everywhere growing, proper, recognizable and generates $X_{\boldsymbol{\sigma}}$.
\end{prop}
\begin{proof}
We start by observing that from Lemma \ref{lem:non_letter-onto_to_letter-onto} we can get that
\begin{equation}\label{eq:prop:make_sequence_reco:1}
\text{$\cA_n \subseteq \cL(X_{\boldsymbol{\sigma}}^{(n)})$ for every $n\in\N$.}
\end{equation}
Let $p_n = \min\{\per(\sigma_{[0,n)}(a)) : a \in \cA_n\}$. Since $\boldsymbol{\sigma}$ is everywhere growing and $X_{\boldsymbol{\sigma}}$ is aperiodic, $\lim_{n\to\infty}p_n = \infty$. Hence, we can contract $\boldsymbol{\sigma}$ in a way such that, for every $n\geq2$, 
\begin{equation*}
\text{$(\mathrm{I}_n)$ $p_n \geq 3|\sigma_{[0,n-1)}|$,}\quad
\text{$(\mathrm{II}_n)$ $\sigma_{[0,n)}$ is $3|\sigma_{[0,n-1)}|$-proper,}
\end{equation*}
For $n\geq 2$, let $U_n = \bigcup_{u,v\in\cA_n^2}[\sigma_{[0,n)}(u.v)]$. Observe that $U_n$ is $|\sigma_{[0,n)}|$-syndetic, has radius $2|\sigma_{[0,n)}|$, is $3|\sigma_{[0,n-1)}|$-proper and, by Remark \ref{rem:proper_implies_separated}, is $p_n$-separated. Thus, by $(\mathrm{I}_n)$, $U$ is $3|\sigma_{[0,n-1)}|$-separated. We can then use Lemma \ref{lem:build_reco_single} with $(X_{\boldsymbol{\sigma}}^{(n)},\sigma_{[0,n)})$ to obtain a letter-onto morphism $\nu_n\colon \cB_n^+ \to \cA_0^+$ and a subshift $Y_n \subseteq \cB_n^\Z$ such that
\begin{enumerate}[label=$(P^\arabic*_n)$]
\item $X_{\boldsymbol{\sigma}} = \bigcup_{k\in\Z}T^k\nu_n(Y_n)$ and $\cB_n \subseteq \cL(Y_n)$,
\item $(Y_n,\nu_n)$ is recognizable with constant $3|\sigma_{[0,n)}|$,
\item $|\nu_n| \leq |\sigma_{[0,n)}|$, $\langle\nu_n\rangle \geq 3|\sigma_{[0,n-1)}|$, and $\nu_n$ is $3|\sigma_{[0,n-1)}|$-proper,
\item $C_{\nu_n}(k,y) = \cR_{U_n}(x)$ for all $x \in X_{\boldsymbol{\sigma}}$ and $\nu_n$-factorization $(k,y)$ of $x$ in $Y_n$.
\end{enumerate}
We write $C_{\nu_n}(x) \coloneqq C_{\nu_n}(k,y)$ if $x \in X_{\boldsymbol{\sigma}}$ and $(k,y)$ is the unique $\nu_n$-factorization of $x$ in $Y_n$. Observe that $U_{n+1} \subseteq U_n$ for $n \geq 2$. Thus, $C_{\nu_{n+1}}(x) = \cR_{U_{n+1}}(x) \subseteq \cR_{U_n}(x) = C_{\nu_n}(x)$ for all $x \in X_{\boldsymbol{\sigma}}$. This, $(P^2_n)$ and $(P^3_{n+1})$ allow us to use Lemma \ref{lem:make_commutative} with $(Y_{n+1},\nu_{n+1})$ and $(Y_n,\nu_n)$ and find a letter-onto and proper morphism $\tau_n\colon \cB_{n+1}^+ \to \cB_{n}^+$ such that $\nu_n\tau_n = \nu_{n+1}$ and $Y_{n} = \bigcup_{k\in\Z}T^k\tau_n(Y_{n+1})$.

Next, we claim that $C_{\nu_n}(x) \supseteq C_{\sigma_{[0,n+1)}}(k,z)$ for all $x \in X_{\boldsymbol{\sigma}}$ and $\sigma_{[0,n+1)}$-factorization $(k,z)$ of $x$ in $X_{\boldsymbol{\sigma}}^{(n+1)}$. Indeed, if $j \in \Z$, then $T^{c_{\sigma_{[0,n+1)},j}(k,z)}x \in [\sigma_{[0,n+1)}(z_{j-1}.z_jz_{j+1})] \subseteq [\sigma_{[0,n)}(a.bc)] \subseteq U_n$, where $a$ is the last letter of $\sigma_{n}(z_{j-1})$ and $bc$ the first two letters of $\sigma_{n}(z_jz_{j+1})$, so $c_{\sigma_{[0,n+1)},j}(k,z) \in \cR_{U_n}(x) = C_{\nu_n}(x)$, as desired.

Thanks to the claim, $(P^2_n)$, $(\mathrm{I}_{n+1})$ and \eqref{eq:prop:make_sequence_reco:1}, we can use Lemma \ref{lem:make_commutative} with $(Y_{n}, \nu_{n})$ and $(X_{\boldsymbol{\sigma}}^{(n+1)},\sigma_{[0,n+1)})$ to obtain a proper morphism $\phi_n\colon \cA_{n+1}^+\to \cB_{n}^+$ such that $\sigma_{[0,n+1)} = \nu_n\phi_n$ and $Y_{n} = \bigcup_{k\in\Z}T^k\phi_n(X_{\boldsymbol{\sigma}}^{(n+1)})$.

Now we can define the morphisms $\tau_1 \coloneqq \nu_2$ and $\phi_1 \coloneqq \nu_2\phi_2$ and the sequences:
\begin{equation*}
\text{$\boldsymbol{\phi} = (\phi_n)_{n\geq1}$, $\boldsymbol{\tau} = (\tau_n)_{n\geq1}$ and $\boldsymbol{\sigma}' = (\sigma_{[0,2)}, \sigma_2,\sigma_3,\dots)_{n\geq2}$.}
\end{equation*}
We are going to prove that $\boldsymbol{\phi}$, $\boldsymbol{\sigma'}$, and $\boldsymbol{\tau}$ are the objects that satisfy the conclusion of the Proposition.

These sequences are letter-onto as each $\nu_n$ and each $\phi_n$ is letter-onto. Next, we show that $\boldsymbol{\phi}$ is a factor. The relation $\phi_1 = \tau_1\phi_2$ follows from the definitions. To prove the other relations, we observe that from the commutative relations for $\tau_n$ and $\phi_n$, we have that
\begin{equation}\label{eq:prop:make_sequence_reco:2}
\nu_{n}\phi_n\sigma_{n+1} = 
\sigma_{[0, n+1)}\sigma_{n+1} = 
\sigma_{[0, n+2)} = 
\nu_{n+1}\phi_{n+1} = \nu_{n}\tau_n\phi_{n+1}.
\end{equation}
In particular, $\nu_{n}\phi_n\sigma_{n+1}(x) = \nu_{n}\tau_n\phi_{n+1}(x)$ for any $x \in X_{\boldsymbol{\sigma}}^{(n+2)}$. Since $\phi_n\sigma_{n+1}(x)$ and $\tau_n\phi_{n+1}(x)$ are both elements of $Y_{n}$ and $(Y_{n},\nu_{n})$ is recognizable, we deduce that $\phi_n\sigma_{n+1}(x) = \tau_n\phi_{n+1}(x)$ for any $x \in X_{\boldsymbol{\sigma}}^{(n+2)}$. Thus, one of the words in $\{\phi_n\sigma_{n+1}(x_0)$, $\tau_n\phi_{n+1}(x_0)\}$ is a prefix of the other. Since $\cA_{n+2} \subseteq \cL(X_{\boldsymbol{\sigma}}^{(n+2)})$, we deduce that, for any $a \in \cA_{n+2}$, one of the words in $\{\tau_n\phi_{n+1}(a)$, $\nu_{n}\phi_n\sigma_{n+1}(a)\}$ is a prefix of the other. But, by \eqref{eq:prop:make_sequence_reco:2}, the words $\nu_{n}\tau_n\phi_{n+1}(a)$ and $\nu_{n}\phi_n\sigma_{n+1}(a)$ have the same length, so $\phi_n\sigma_{n+1}(a)$ must be equal to $\tau_n\phi_{n+1}(a)$ for every $n\geq2$. This proves that $\phi_n\sigma_{n+1} = \tau_n\phi_{n+1}$ for every $n\geq2$ and that $\boldsymbol{\phi}\colon\boldsymbol{\sigma'}\to\boldsymbol{\tau}$ is a factor.

The following commutative diagram, valid for all $n\geq2$, summarizes the construction so far:
\begin{equation*}
\begin{tikzcd}[column sep=11ex]
	\cA_{n+2}^+ \arrow{r}{\sigma_{n+1}}
	\arrow{d}[swap]{\phi_{n+1}} 
	& \cA_{n+1}^+ \arrow{r}{\sigma_{[0, n+1)}} 
	\arrow{d}[yshift=-0.4ex]{\phi_n}
	& \cA_0^+ \\
	\cB_{n+1}^+  \arrow{r}[swap]{\tau_n} 
	\arrow[bend left=4]{urr}[xshift=-8ex,yshift=-2.25ex]{\nu_{n+1}}
	& \cB_{n}^+ \arrow{ur}[swap]{\nu_{n}} &
\end{tikzcd}
\end{equation*}
As shown in the diagram, we have that $\nu_{n}\tau_n = \nu_{n+1}$ for $n\geq2$. Thus, $\tau_1\tau_2\cdots\tau_n = \nu_{n+1}$, and hence $\langle \tau_1\tau_2\cdots\tau_n \rangle \geq \langle \nu_{n+1} \rangle \geq p_n \to_{n\to\infty} \infty$. Therefore, $\boldsymbol{\tau}$ is everywhere growing. Also, by using Lemma \ref{lem:recognizability_of_composition} with $(Y_{n},\nu_{n}) = (Y_{n},\tau_1\tau_2\cdots\tau_{n-1})$, we deduce that $(Y_{n},\tau_{n-1})$ is recognizable for every $n\geq2$, which implies that $\boldsymbol{\tau}$ is recognizable. Finally, as each $\tau_n$ is proper, $\boldsymbol{\tau}$ is proper.
\end{proof}

\section*{Acknowledgement}
This research was partially supported by grant ANID-AFB 170001. The first author thanks Doctoral Felowship CONICYT-PFCHA/Doctorado Nacional/2020-21202229


\printbibliography

\end{document}

%% file: defs.tex
\newtheorem{theo}{Theorem}[section]
\newtheorem{prop}[theo]{Proposition}
\newtheorem{lem}[theo]{Lemma}
\newtheorem{cor}[theo]{Corollary}

\newtheorem{rem}[theo]{Remark}
\newtheorem{defi}[theo]{Definition}
\newtheorem{ques}[theo]{Question}
\newtheorem{innercustomthm}{Theorem}
\newtheorem{claim}{Claim}[theo]
\newenvironment{ctheo}[1]
  {\renewcommand\theinnercustomthm{#1}\innercustomthm}
  {\endinnercustomthm}
\newtheorem{innercustomcor}{Corollary}
\newenvironment{ccor}[1]
  {\renewcommand\theinnercustomcor{#1}\innercustomcor}
  {\endinnercustomcor}
\newenvironment{claimproof}[1]{\par\smallskip\noindent{\em Proof.}\space#1}{\hfill $\square$\par\smallskip}

\def\per{{\mathrm{per}}}

\def\Z{{\mathbb Z}}

\def\N{{\mathbb N}}

\def\cA{{\mathcal A}}
\def\cB{{\mathcal B}}
\def\cC{{\mathcal C}}
\def\cD{{\mathcal D}}

\def\cL{{\mathcal L}}

\def\cP{{\mathcal P}}
\def\cR{{\mathcal R}}
\def\cS{{\mathcal S}}
\def\cT{{\mathcal T}}

\def\cW{{\mathcal W}}

%% file: biblio.bib
@article {DDPM20,
    AUTHOR = {Donoso, Sebasti\'{a}n and Durand, Fabien and Maass, Alejandro and Petite, Samuel},
     TITLE = {Interplay between finite topological rank minimal {C}antor
              systems, {$\mathcal{S}$}-adic subshifts and their complexity},
   JOURNAL = {Trans. Amer. Math. Soc.},
  FJOURNAL = {Transactions of the American Mathematical Society},
    VOLUME = {374},
      YEAR = {2021},
    NUMBER = {5},
     PAGES = {3453--3489},
      ISSN = {0002-9947},
   MRCLASS = {37B10 (68R15)},
  MRNUMBER = {4237953},
       DOI = {10.1090/tran/8315},
       URL = {https://doi.org/10.1090/tran/8315},
}

@article{BSTY17,
    AUTHOR = {Berth\'{e}, Val\'{e}rie and Steiner, Wolfgang and Thuswaldner, J\"{o}rg M.
              and Yassawi, Reem},
     TITLE = {Recognizability for sequences of morphisms},
   JOURNAL = {Ergodic Theory Dynam. Systems},
  FJOURNAL = {Ergodic Theory and Dynamical Systems},
    VOLUME = {39},
      YEAR = {2019},
    NUMBER = {11},
     PAGES = {2896--2931},
      ISSN = {0143-3857},
   MRCLASS = {37B50},
  MRNUMBER = {4015135},
       DOI = {10.1017/etds.2017.144},
       URL = {https://doi.org/10.1017/etds.2017.144},
}

@article{EM, 
title={On the automorphism group of minimal $\mathcal {S}$-adic subshifts of finite alphabet rank}, 
DOI={10.1017/etds.2021.64}, 
journal={Ergodic Theory and Dynamical Systems}, 
publisher={Cambridge University Press}, 
author={Espinoza, Bastián and Maass, Alejandro}, 
year={2021}, 
pages={1–23}
}

@article {DM,
AUTHOR = {Downarowicz, Tomasz and Maass, Alejandro},
TITLE = {Finite-rank {B}ratteli-{V}ershik diagrams are expansive},
JOURNAL = {Ergodic Theory Dynam. Systems},
FJOURNAL = {Ergodic Theory and Dynamical Systems},
VOLUME = {28},
YEAR = {2008},
NUMBER = {3},
PAGES = {739--747},
ISSN = {0143-3857},
MRCLASS = {37B10 (37B05)},
MRNUMBER = {2422014},
MRREVIEWER = {A. I. Danilenko},
DOI = {10.1017/S0143385707000673},
URL = {https://doi.org/10.1017/S0143385707000673},
}

@book{auslander,
title={Minimal Flows and Their Extensions},
author={Auslander, J.},
isbn={9780080872643},
series={ISSN},
url={https://books.google.cl/books?id=e3wFvPvpWvwC},
year={1988},
publisher={Elsevier Science}
}

@book{handbook_words,
title={Handbook of Formal Languages: Volume 1. Word, Language, Grammar},
author={Rozenberg, G. and Salomaa, A.},
isbn={9783540604204},
lccn={lc96047134},
series={Handbook of Formal Languages},
url={https://books.google.cl/books?id=yQ59ojndUt4C},
year={1997},
publisher={Springer}
}

@article{GH2020, 
title={On topological rank of factors of Cantor minimal systems}, DOI={10.1017/etds.2021.62}, 
journal={Ergodic Theory and Dynamical Systems}, 
publisher={Cambridge University Press}, 
author={Golestani, Nasser and Hosseini, Maryam}, 
year={2021}, 
pages={1–24}
}

@article {AEG2015,
    AUTHOR = {Amini, Massoud and Elliott, George A. and Golestani, Nasser},
     TITLE = {The category of {B}ratteli diagrams},
   JOURNAL = {Canad. J. Math.},
  FJOURNAL = {Canadian Journal of Mathematics. Journal Canadien de
              Math\'{e}matiques},
    VOLUME = {67},
      YEAR = {2015},
    NUMBER = {5},
     PAGES = {990--1023},
      ISSN = {0008-414X},
   MRCLASS = {46L35 (18A30 46L05 46L80 46M15)},
  MRNUMBER = {3391730},
MRREVIEWER = {Tobias W. Fritz},
       DOI = {10.4153/CJM-2015-001-8},
       URL = {https://doi.org/10.4153/CJM-2015-001-8},
}

@article {eigenvalues,
    AUTHOR = {Durand, Fabien and Frank, Alexander and Maass, Alejandro},
     TITLE = {Eigenvalues of minimal {C}antor systems},
   JOURNAL = {J. Eur. Math. Soc. (JEMS)},
  FJOURNAL = {Journal of the European Mathematical Society (JEMS)},
    VOLUME = {21},
      YEAR = {2019},
    NUMBER = {3},
     PAGES = {727--775},
      ISSN = {1435-9855},
   MRCLASS = {37B05 (37B10 54H20)},
  MRNUMBER = {3908764},
       DOI = {10.4171/JEMS/849},
       URL = {https://doi.org/10.4171/JEMS/849},
}

@article{invariant_measures,
 ISSN = {00029947},
 URL = {http://www.jstor.org/stable/23513496},
 author = {S. Bezuglyi and J. Kwiatkowski and K. Medynets and B. Solomyak},
 journal = {Transactions of the American Mathematical Society},
 number = {5},
 pages = {2637--2679},
 publisher = {American Mathematical Society},
 title = {Finite rank Bratteli diagrams: structure of invariant measures},
 volume = {365},
 year = {2013}
}

@article {dimension_group_bounded,
    AUTHOR = {Giordano, T. and Handelman, D. and Hosseini, M.},
     TITLE = {Orbit equivalence of {C}antor minimal systems and their
              continuous spectra},
   JOURNAL = {Math. Z.},
  FJOURNAL = {Mathematische Zeitschrift},
    VOLUME = {289},
      YEAR = {2018},
    NUMBER = {3-4},
     PAGES = {1199--1218},
      ISSN = {0025-5874},
   MRCLASS = {37A20 (19K14 37B05)},
  MRNUMBER = {3830245},
MRREVIEWER = {Lewis Bowen},
       DOI = {10.1007/s00209-017-1994-9},
       URL = {https://doi.org/10.1007/s00209-017-1994-9},
}

@article {intervals_exchanges_are_bratteli,
    AUTHOR = {Gjerde, Richard and Johansen, {\O}rjan},
     TITLE = {Bratteli-{V}ershik models for {C}antor minimal systems
              associated to interval exchange transformations},
   JOURNAL = {Math. Scand.},
  FJOURNAL = {Mathematica Scandinavica},
    VOLUME = {90},
      YEAR = {2002},
    NUMBER = {1},
     PAGES = {87--100},
      ISSN = {0025-5521},
   MRCLASS = {37B10 (37A45)},
  MRNUMBER = {1887096},
MRREVIEWER = {Peter G. Danilaev},
       DOI = {10.7146/math.scand.a-14363},
       URL = {https://doi.org/10.7146/math.scand.a-14363},
}

@article {herman,
    AUTHOR = {Herman, Richard H. and Putnam, Ian F. and Skau, Christian F.},
     TITLE = {Ordered {B}ratteli diagrams, dimension groups and topological
              dynamics},
   JOURNAL = {Internat. J. Math.},
  FJOURNAL = {International Journal of Mathematics},
    VOLUME = {3},
      YEAR = {1992},
    NUMBER = {6},
     PAGES = {827--864},
      ISSN = {0129-167X},
   MRCLASS = {46L99 (46L55 46L80 54H20 58F11)},
  MRNUMBER = {1194074},
MRREVIEWER = {Yiu Tung Poon},
       DOI = {10.1142/S0129167X92000382},
       URL = {https://doi.org/10.1142/S0129167X92000382},
}

@article{durand_2000, title={Linearly recurrent subshifts have a finite number of non-periodic subshift factors}, volume={20}, DOI={10.1017/S0143385700000584}, number={4}, journal={Ergodic Theory and Dynamical Systems}, publisher={Cambridge University Press}, author={Durand, Fabien}, year={2000}, pages={1061–1078}}

@ARTICLE{E20,
       author = {{Espinoza}, Basti{\'a}n},
        title = "{On symbolic factors of S-adic subshifts of finite alphabet rank}",
      journal = {arXiv e-prints},
         year = 2020,
        month = aug,
archivePrefix = {arXiv},
       eprint = {2008.13689v2},
 primaryClass = {math.DS},
       adsurl = {https://ui.adsabs.harvard.edu/abs/2020arXiv200813689E}
}

@article {DHS99,
    AUTHOR = {Durand, F. and Host, B. and Skau, C.},
     TITLE = {Substitutional dynamical systems, {B}ratteli diagrams and
              dimension groups},
   JOURNAL = {Ergodic Theory Dynam. Systems},
  FJOURNAL = {Ergodic Theory and Dynamical Systems},
    VOLUME = {19},
      YEAR = {1999},
    NUMBER = {4},
     PAGES = {953--993},
      ISSN = {0143-3857},
   MRCLASS = {46L55 (37B99 46L80 54H20)},
  MRNUMBER = {1709427},
MRREVIEWER = {A. I. Danilenko},
       DOI = {10.1017/S0143385799133947},
       URL = {https://doi.org/10.1017/S0143385799133947},
}

@article {toeplitz,
    AUTHOR = {Gjerde, Richard and Johansen, \O rjan},
     TITLE = {Bratteli-{V}ershik models for {C}antor minimal systems:
              applications to {T}oeplitz flows},
   JOURNAL = {Ergodic Theory Dynam. Systems},
  FJOURNAL = {Ergodic Theory and Dynamical Systems},
    VOLUME = {20},
      YEAR = {2000},
    NUMBER = {6},
     PAGES = {1687--1710},
      ISSN = {0143-3857},
   MRCLASS = {46L55 (28D20 37A55)},
  MRNUMBER = {1804953},
MRREVIEWER = {Sergey Neshveyev},
       DOI = {10.1017/S0143385700000948},
       URL = {https://doi.org/10.1017/S0143385700000948},
}
